\documentclass[11pt]{amsart}
\usepackage{a4wide}
\usepackage{amssymb}
\usepackage{amsmath}
\usepackage{amsthm}
\usepackage{mathrsfs}
\usepackage{tikz}
\usepackage{amscd}
\usepackage{hyperref}
\usepackage{color}

\newcommand{\id}{\operatorname{id}}

\newcommand{\ind}{\operatorname{ind}}

\newcommand{\res}{\operatorname{res}}
\newcommand{\ho}{\operatorname{Hom}}

\newcommand{\g}[2]{\operatorname{GL}_{#1}(#2)}
\newcommand{\diag}{\operatorname{diag}}

\newcommand{\End}{\operatorname{End}}

\newcommand{\Rad}{\operatorname{Rad}}

\newtheorem{theorem}{Theorem}[section]

\newtheorem{lemma}[theorem]{Lemma}
\newtheorem{remark}[theorem]{Remark}

\title{On typical representations for depth-zero components of split
  classical groups} \author{{Amiya Kumar Mondal} and
  {Santosh Nadimpalli}} \date{\today}
\begin{document}
\begin{abstract}
  Let ${\bf G}$ be a split classical group over a non-Archimedean
  local field $F$ with the cardinality of the residue field $q_F>5$.  Let
  $M$ be the group of $F$-points of a Levi factor of a proper
  $F$-parabolic subgroup of ${\bf G}$. Let $[M, \sigma_M]_M$ be an
  inertial class such that $\sigma_M$ contains a depth-zero
  Moy--Prasad type of the form $(K_M, \tau_M)$, where $K_M$ is a
  hyperspecial maximal compact subgroup of $M$. Let $K$ be a
  hyperspecial maximal compact subgroup of ${\bf G}(F)$ such that $K$
  contains $K_M$.  In this article, we classify $\mathfrak{s}$-typical
  representations of $K$. In particular, we show that the
  $\mathfrak{s}$-typical representations of $K$ are precisely the
  irreducible subrepresentations of $\ind_J^K\lambda$, where
  $(J, \lambda)$ is a level-zero $G$-cover of $(K\cap M, \tau_M)$.
\end{abstract}{\let\thefootnote\relax}
\maketitle
\section{Introduction}
Let $F$ be a non-Archimedean local field with ring of integers
$\mathfrak{o}_F$. Let $\mathfrak{p}_F$ be the maximal ideal of
$\mathfrak{o}_F$. Let $k_F$ be the residue field of $\mathfrak{o}_F$
and we assume that $k_F$ has cardinality $q_F>5$. Let ${\bf G}$ be any
reductive algebraic group over $F$, and let $G$ be the group of
$F$-rational points of ${\bf G}$. Let $K$ be any maximal compact
subgroup of $G$. All representations in this article are defined over
complex vector spaces.

Let $(M, \sigma_M)$ be a pair consisting of a Levi factor $M$ of an
$F$-parabolic subgroup of $G$, and a cuspidal representation
$\sigma_M$ of $M$. Recall that two such pairs $(M_1, \sigma_{M_1})$
and $(M_2, \sigma_{M_2})$ are called {\it inertially equivalent} if
there exists an element $g\in G$ such that
$$M_1=gM_2g^{-1}~{\rm and}~\sigma_{M_1}\simeq \sigma_{M_2}^g\otimes\chi,$$
where $\chi$ is an unramified character of $M_1$. Equivalence classes
for this relation are called {\it inertial classes}. The inertial
class containing the pair $(M, \sigma_M)$ is denoted by
$[M, \sigma_M]_G$ (or by $[M, \sigma_M]$ if $G$ is clear from the
context).  The set of inertial classes of $G$ is denoted by
$\mathcal{B}(G)$. An inertial class of the form $[G, \sigma]_G$ is
called a {\it cuspidal inertial class of $G$}.

Let $\mathcal{R}(G)$ be the category of smooth representations of
$G$. Let $\mathfrak{s}=[M, \sigma_M]_G$ be an inertial class of
$G$, and let $\mathcal{R}_\mathfrak{s}(G)$ be the full subcategory of
$\mathcal{R}(G)$ consisting of smooth $G$-representations whose
irreducible subquotients occur as subquotients of
$i_P^{G}(\sigma_M\otimes\chi)$, where $P$ is an $F$-parabolic subgroup
such that $M$ is a Levi factor of $P$ and $\chi$ is an unramified
character of $M$. Here, the functor $i_P^G$ denotes
  the normalised parabolic induction. Bernstein in the article
\cite{le_centre_bernstein} showed that the category $\mathcal{R}(G)$
can be decomposed as
$$\mathcal{R}(G)=
\prod_{\mathfrak{s}\in \mathcal{B}(G)}\mathcal{R}_\mathfrak{s}(G).$$
The category $\mathcal{R}_\mathfrak{s}(G)$ is indecomposable.  In
particular, every smooth representation of $G$ can be written as a
direct sum of subrepresentations which belong to
$\mathcal{R}_\mathfrak{s}(G)$. The category
$\mathcal{R}_\mathfrak{s}(G)$ is called the {\it Bernstein component}
associated to $\mathfrak{s}$. 

Based on extensive examples for ${\rm GL}_n$, ${\rm SL}_n$, it turns
out that for a given indecomposable block
$\mathcal{R}_\mathfrak{s}(G)$, there is a natural set of irreducible
smooth representations of $K$ called $\mathfrak{s}$-typical
representations: if an $\mathfrak{s}$-typical representation of $K$
occurs in an irreducible smooth representation $\pi$ of $G$ then,
$\pi$ belongs to $\mathcal{R}_\mathfrak{s}(G)$. In this article, when
$K$ is hyperspecial, we classify $\mathfrak{s}$-typical
representations of $K$, for depth-zero inertial classes $\mathfrak{s}$
of split classical groups. We refer to the articles
\cite{Henniart-gl_2}, \cite{Paskunas-uniqueness}, \cite{forum_gl3_article},
\cite{level_zero_gl_n_types}, \cite{latham_2}, and \cite{Latham2018}
for some earlier works. We will now try to make these notations
precise and describe our main theorem.

Theory of types, developed by Bushnell--Kutzko, describes the category
$\mathcal{R}_\mathfrak{s}(G)$ in terms of modules over Hecke
algebras. We refer to \cite{Smoothrepcptgr} for a systematic
treatment.  In particular, the formalism aims to construct a pair
$(J_\mathfrak{s}, \lambda_\mathfrak{s})$ consisting of a compact open
subgroup $J_\mathfrak{s}$ of $G$, and an irreducible smooth
representation $\lambda_\mathfrak{s}$ of $J_\mathfrak{s}$ such that,
for any irreducible smooth representation $\pi$ of $G$,
\begin{equation}
\ho_{J_\mathfrak{s}}(\lambda_\mathfrak{s}, \pi)\neq 0
 \ \text{if and only if}\  \pi \in
\mathcal{R}_\mathfrak{s}(G).
\end{equation}
Such a pair $(J_\mathfrak{s}, \lambda_\mathfrak{s})$ is called a {\it
  type for $\mathfrak{s}$} or an {\it $\mathfrak{s}$-type}. 

A type $(J_\mathfrak{s},\lambda_\mathfrak{s})$, for an inertial class
$\mathfrak{s}=[M, \sigma_M]_G$, is generally constructed in two
steps. First, a type $(J_\mathfrak{t}, \lambda_\mathfrak{t})$ is
constructed for the cuspidal inertial class
$\mathfrak{t}=[M, \sigma_M]_M$. For the inertial class
$[M, \sigma_M]_G$, a type $(J_\mathfrak{s},\lambda_\mathfrak{s})$ is
then constructed as a $G$-cover of
$(J_\mathfrak{t},\lambda_\mathfrak{t})$, in the sense of \cite[Section
8]{Smoothrepcptgr}. In particular, for any $F$-parabolic subgroup $P$
of $G$ such that $M$ is a Levi factor of $P$, a $G$-cover
$(J_\mathfrak{s}, \lambda_\mathfrak{s})$ has Iwahori decomposition
with respect to the pair $(P, M)$ i.e., $J_\mathfrak{s}\cap M$ is
equal to $J_\mathfrak{t}$,
$$\res_{J_\mathfrak{s}\cap M}\lambda_\mathfrak{s}=\lambda_\mathfrak{t},$$
and the groups $J_\mathfrak{s}\cap U$ and $J_\mathfrak{s}\cap \bar{U}$
are both contained in the kernel of $\lambda_\mathfrak{s}$. Here $U$
is the unipotent radical of $P$, and $\bar{U}$ is the unipotent
radical of the opposite parabolic subgroup of $P$ with respect to $M$.

Types $(J_\mathfrak{s}, \lambda_\mathfrak{s})$ are now constructed for
many classes of reductive groups $G$. There are several constructions
leading to different pairs $(J_\mathfrak{s}, \lambda_\mathfrak{s})$ as
types for $\mathfrak{s}$. These types contain important arithmetic
information. For ${\rm GL}_n(F)$, Bushnell and Kutzko
\cite{Orrangebook} constructed a set of types, which they called {\it
  maximal types}, for any cuspidal component. Later
in the article \cite{Bushnell-kutzko-Semisimpletypes}, they
constructed explicit $G$-covers for these maximal types. For
${\rm SL}_n(F)$, similar constructions are due to Bushnell--Kutzko and
Goldberg--Roche (see \cite{admissible_dual_sl_n_1},
\cite{admissible_dual_sl_n_2}, \cite{types_in_sl_n} and
\cite{hecke_algebras_and_sl_n_types}), for inner forms of ${\rm GL}_n$
by S\'echerre and Stevens (see
\cite{cuspidals_gl(D)_4} and \cite{semisimpletypes_gl(D)}), for
${\rm Sp}_4(F)$ by Blasco and Blondel in \cite{tsstypes_sp_4_bb} and
\cite{hecke_algebras_sp_4_bb}. Types for inertial classes of the form
$[T, \chi]$, where $T$ is a maximal split torus are constructed by
Roche \cite{AllanRoche}. For an arbitrary connected reductive group
and depth-zero components, types are constructed by
Morris, Moy and Prasad in \cite{level_zero_G-types}
and \cite{Jac_fnct_moy_prasad} respectively. For classical groups
(with $p$ odd), these construction are due to Stevens
\cite{cusps_classical}, and by Miyauchi and Stevens
\cite{miyauchi_stevens}.

Let $K$ be a maximal compact subgroup of $G$, and let $\mathfrak{s}$
be an inertial class of $G$. An irreducible smooth representation
$\tau$ of $K$ is called {\it $\mathfrak{s}$-typical} if every
irreducible smooth representation $\pi$ of $G$ such that
$\ho_{K}(\tau_, \pi)\neq 0$ is in $\mathcal{R}_\mathfrak{s}(G)$. This
notion weakens that of $\mathfrak{s}$-type introduced by Bushnell and
Kutzko: $\tau$ is an $\mathfrak{s}$-type if it is
$\mathfrak{s}$-typical and ${\rm Hom}_K(\tau, \pi)\neq 0$, for all
irreducible smooth representations $\pi$ in
$\mathcal{R}_\mathfrak{s}(G)$. An irreducible smooth representation
$\tau$ of $K$ is called {\it atypical} if $\tau$ is not an
$\mathfrak{s}$-typical representation for any
$\mathfrak{s}\in \mathcal{B}(G)$.  Let
$(J_\mathfrak{s}, \lambda_\mathfrak{s})$ be an $\mathfrak{s}$-type
such that $J_\mathfrak{s}\subseteq K$.  Then Frobenius reciprocity
shows that any irreducible subrepresentation of
\begin{equation}\label{types}
  \ind_{J_\mathfrak{s}}^K\tau_\mathfrak{s}
\end{equation}
is $\mathfrak{s}$-typical. In general, the representation
\eqref{types} is not irreducible, and hence, the isomorphism classes
of $\mathfrak{s}$-typical representations of $K$ are not necessarily
unique. In the interest of arithmetic applications, it is important
to understand the existence and classification of
$\mathfrak{s}$-typical representations of $K$.

The representation theory of maximal compact subgroups of $p$-adic
groups is quite involved. For example, a parametrisation of all
irreducible smooth representations for $K=\g{n}{\mathfrak{o}_F}$ is
not yet known.  In this regard, it is interesting to understand
irreducible smooth representations of $K$ in terms of
the Bernstein decomposition of $G$. Precisely, for
any finite set of inertial classes $\mathcal{S}$ of $G$, one wants to
understand those irreducible smooth representations $\tau$ of $K$ such
that, for an irreducible smooth representation $\pi$ of $G$,
$$\ho_{K}(\tau, \pi)\neq 0 \  \Rightarrow \  \pi \in
\mathcal{R}_\mathfrak{s}(G),\ \text{for some}\ \mathfrak{s}\ \in
\mathcal{S}.$$ This article belongs to this theme.

We now state the main results of this article.  Let $(W, q)$ be a pair
consisting of an $F$-vector space $W$, and a non-degenerate
alternating or symmetric $F$-bilinear form $q$ on $W$.  Let $G$ be the
group of $F$-points of ${\bf G}$--the connected component of the
isometry group associated to the pair $(W, q)$. {\it We assume that
  ${\bf G}$ is an $F$-split group}. For any parahoric subgroup
$\mathcal{K}$ of $G$ we denote by $\mathcal{K}^+$ the
pro-$p$-unipotent radical of $\mathcal{K}$. Let $\mathfrak{t}$ be an
inertial class $[M, \sigma_M]_M$ such that $\sigma_M^{K_M^+}\neq 0$,
for some maximal parahoric subgroup $K_M$ of $M$. The representation
$\sigma_M$ is called a {\it depth-zero} cuspidal representation of $M$
and the inertial class $\mathfrak{t}$ is called a {\it depth-zero
  inertial class}. Any irreducible $K_M$-subrepresentation of
$\sigma_M^{K_M^+}$ is the inflation of a cuspidal representation of
the finite reductive group $K_M/K_M^+$. Let $\tau_M$ be an irreducible
$K_M$-subrepresentation of $\sigma_M^{K_M^+}$.  The pair
$(K_M, \tau_M)$ is called an {\it unrefined minimal $K$-type} by Moy
and Prasad (see \cite[Definition 5.1]{unrefined_minimal_types}).  When
$K_M$ is a hyperspecial maximal compact subgroup, the pair
$(K_M, \tau_M)$ is also a $[M, \sigma_M]_M$-type in the sense of
Bushnell and Kutzko; in this case, we simply call the pair $(K_M, \tau_M)$ a
{\it depth-zero} type.

Assume that $K_M$ is a hyperspecial maximal compact subgroup of
$M$. Let $K$ be a hyperspecial maximal compact subgroup of $G$ such
that $K_M\subset K$.  Let $P$ be a parabolic subgroup of $G$ such that
$M$ is a Levi factor of $P$. Let $P(1)$ be the group $(P\cap
K)K^+$. Note that the group $P(1)$ is a parahoric
  subgroup of $G$, and we have $P(1)\cap M=K_M$. The
representation $\tau_M$ of $K_M$ extends as a representation of $P(1)$
such that $P(1)\cap U$ and $P(1)\cap \bar{U}$ are contained in the
kernel of this extension. Here, $U$ is the unipotent radical of $P$
and $\bar{U}$ is the unipotent radical of the opposite parabolic
subgroup of $P$ with respect to $M$.  With these notations, our main
result can be stated as follows:
 \begin{theorem}\label{thm intro} 
   Let $\mathfrak{s}=[M, \sigma_M]_G$ be an inertial
     class such that $M\neq G$. Let $K_M$ be a hyperspecial maximal
   compact subgroup of $M$. Assume that $\sigma_M^{K_M^+}\neq 0$, and
   let $\tau_M$ be an irreducible $K_M$-subrepresentation of
   $\sigma_M^{K_M^+}$. Let $K$ be a hyperspecial maximal compact
   subgroup of $G$ such that $K_M\subseteq K$. Then
   $\mathfrak{s}$-typical representations of $K$ are exactly the
   subrepresentations of $\ind_{P(1)}^K\tau_M$.
\end{theorem}
Let $G$ be the group of $F$-points of a reductive algebraic group
defined over $F$. For the depth-zero inertial classes of the form
$\mathfrak{s}=[G, \sigma]_G$, and $K$ is any maximal compact subgroup,
Latham \cite{latham_2} showed that an $\mathfrak{s}$-typical
representation of $K$, if it exists, is unique. We will apply this
result for split classical groups. However, for the present purposes
of this article, we only need to consider hyperspecial maximal compact
subgroups (see Lemma \ref{unicity_classical_cusps}).

Let ${\bf T}$ be a maximal split torus of ${\bf G}$
  defined over $F$. Using a Witt--Basis, we identify ${\bf T}(F)$ with
  the following sub-torus of the diagonal torus of ${\rm GL}(W)$:
$$\{\diag(t_1,\dots,t_1^{-1}):t_i\in F^\times,
1\leq i\leq n\}.$$ Let $\chi$ be a character of ${\bf T}(F)$, and let
$$\chi(\diag(t_1, \dots,t_1^{-1}))=\chi_1(t_1)\cdots\chi_n(t_n),$$ where
$\chi_i$ is a character of $F^\times$, for $1\leq i\leq n$. The
inertial class $[{\bf T}(F), \chi]_G$ is called a {\it toral inertial
class}. For any character $\eta$ of $F^\times$, let $l(\eta)$ be the
least positive integer $k$ such that $1+\mathfrak{p}_F^k$ is contained
in the kernel of $\eta$. In this article, we assume that
\begin{equation}\label{intro_7_cond}
  l(\chi_i)\neq l(\chi_j),\ \text{for}\  1\leq i\neq j\leq n.
\end{equation}

Let $K$ be a hyperspecial maximal compact subgroup of ${\bf G}$ such
that ${\bf T}(F)\cap K$ is the maximal compact subgroup of
${\bf T}(F)$. The proof of Theorem \ref{thm intro} can also be
extended to obtain a classification of $\mathfrak{s}$-typical
representations of $K$. In Section \ref{princ}, we describe Roche's
construction of a $G$-cover $(J_\chi, \chi)$ for the pair
$({\bf T}(F)\cap K, \res_{{\bf T}(F)\cap K}\chi)$ (see \cite[Section
2,3]{AllanRoche}). This construction depends on the choice of a
pinning. It is possible to choose a pinning such that
$J_\chi\subset K$. We prove the following theorem for the toral
inertial class $[{\bf T}(F), \chi]$.
\begin{theorem}
  Let $K$ be any hyperspecial maximal compact subgroup of $G$. Let
  ${\bf T}$ be any maximal split torus of ${\bf G}$ defined over
  $F$. Assume that $K\cap {\bf T}(F)$ is the maximal compact subgroup
  of ${\bf T}(F)$. Let $\chi$ be a character of ${\bf T}(F)$ which satisfies
  the condition \eqref{intro_7_cond}. Then $[{\bf T}(F),\chi]_G$-typical
  representations of $K$ are exactly the subrepresentations of
  ${\rm ind}_{J_\chi}^K \chi$.
\end{theorem}
\section{Notations}
Let $F$ be a non-Archimedean local field with ring of integers
$\mathfrak{o}_F$. Let $\mathfrak{p}_F$ be the maximal ideal of
$\mathfrak{o}_F$ with residue field
$k_F=\mathfrak{o}_F/\mathfrak{p}_F$. Let $q_F$ be the cardinality of
$k_F$. In this article, we assume that $q_F>5$. Let $\varpi_F$ be an
uniformiser of $F$.  For any $F$-algebraic group ${\bf H}$, we denote
by $H$, the group ${\bf H}(F)$. The group $H$ is considered as a
topological group whose topology is induced from $F$.

Let ${\bf G}$ be any reductive algebraic group over $F$. For any
closed subgroup $H$ of $G$ and a smooth representation $\sigma$ of
$H$, we denote by $\ind_H^G\sigma$, the compactly induced
representation from $H$ to $G$. For any parabolic subgroup $P$ of $G$
and $\sigma$ any smooth representation of a Levi factor $M$ of $P$, we
denote by $i_{P}^G\sigma$, the normalised parabolically induced
representation of $G$. For any representations $\rho_1$ and $\rho_2$
of the groups $G_1$ and $G_2$ respectively, we denote by
$\rho_1\boxtimes\rho_2$, the tensor product representation of the
group $G_1\times G_2$.

Let $(V, q)$ be any pair consisting of a vector space $V$ over a field
$k$, and a $k$-bilinear form $q$ on $V$. We denote by $G(V,q)$ (or by
$G(V)$ when $q$ is clear from the context), the group of $k$-points of
the connected component of the isometry group of the pair $(V,q)$.
\section{Preliminaries}\label{preli}
Let $\epsilon\in\{\pm 1\}$, and let $W$ be an $F$-vector space with
a non-degenerate $F$-bilinear form $q$ such that
$$q(w_1, w_2)=\epsilon q(w_2, w_1),\  \text{for}\ w_1, w_2\in W.$$ 
Let $W^+$ be any maximal totally isotropic subspace of
$W$. Let $$(w_1, w_2,\dots,w_n)$$ be a basis of $W^+$. There exists a
maximal totally isotropic subspace $W^-$ with basis
$$(w_{-1}, w_{-2},\dots,w_{-n})$$ such that
\begin{equation}
  q(w_i, w_j)=0,\ \text{for}\ -n\leq i\neq -j\leq n,\ \text{and}\
 q(w_i, w_{-i})=1,\ \text{for}\ 1\leq i\leq n.
 \end{equation}
 The space $W^+\oplus W^-$ is a hyperbolic subspace of $W$. Let
 $(W^+\oplus W^-)\perp W_0$ be a Witt--decomposition of $W$. Note that
 $W_0$ is an anisotropic subspace of $W$. {\bf In this article, we
   assume that $\dim_FW_0\leq 1$.} Let $w_0$ be any non-zero vector in
 $W_0$, if $W_0\neq \{0\}$. The tuple of vectors
\begin{eqnarray}\label{standard_basis}
 B:=\begin{cases}
 (w_n, w_{n-1}, \cdots, w_1, w_{-1}, w_{-2}, \cdots, w_{-n})
    &\text{if $\dim(W)=2n$}\\
    (w_n, w_{n-1}, \cdots, w_1, w_0, w_{-1}, w_{-2}, \cdots, w_{-n})
    &\text{if $\dim(W)=2n+1$.}
      \end{cases}
 \end{eqnarray}
 is a basis of the space $W$. Any tuple of vectors as in $B$ is called a
 {\it standard basis} of $W$. Let $N$ be the cardinality of the basis
 $B$.  Let ${\bf G}/F$ be the connected component of the isometry group
 associated to the pair $(W,q)$. The group ${\bf G}$ is an $F$-split
 semisimple group.  Any standard basis $B$ gives the
 following isomorphism
\begin{eqnarray}
{\bf G}\simeq\begin{cases}
 {\bf SO}_{2n}/F &  \text{if}\  \epsilon=1,\
 \text{and}\  N=2n\\
 {\bf SO}_{2n+1}/F& \text{if}\ \epsilon=1\ \text{and}\ N=2n+1,\\
  {\bf Sp}_{2n}/F& \text{if} \ \epsilon=-1.
\end{cases}
\end{eqnarray}

Given any maximal split torus ${\bf T}$ (defined over $F$) of ${\bf
  G}$, there exists a standard basis $B=(w_i: -n\leq i\leq n)$ of
$W$ such that $T$ is the $G$-stabilizer of the decomposition
$$W=Fw_{n}\oplus Fw_{n-1}\oplus\cdots\oplus Fw_{-n+1}\oplus Fw_{-n}.$$
Conversely, any standard basis $B$ gives rise to a maximal split torus
${\bf T}$ in ${\bf G}$ such that $T$ is the $G$-stabilizer of the
decomposition as above. We say that the torus ${\bf T}$ is associated
to the standard basis $B$.

A {\it lattice chain} $\Lambda$ is a function from $\mathbb{Z}$
to the set of lattices in $W$ which satisfies the following conditions:
\begin{enumerate}
\item $\Lambda(j)\subsetneq \Lambda(i)$, for $i<j$, and
\item there exists an integer $e(\Lambda)$ such that
  $\Lambda(i+e(\Lambda))=\mathfrak{p}_F\Lambda(i)$, for all $i\in \mathbb{Z}$. 
\end{enumerate}  
 Given any lattice $\mathcal{L}$, let
  $\mathcal{L}^\#$ be the lattice
$$\mathcal{L}^\#:=\{w\in W\ | \ q(v, \mathcal{L})\subset \mathfrak{p}_F\}.$$
Let $\Lambda^\#$ be the lattice chain defined by setting
$$\Lambda^\#(i)=\Lambda(-i)^\#,~{\rm for~all}~i\in \mathbb{Z}.$$ A lattice
chain $\Lambda$ is called {\it self-dual} if there exists $d\in
\mathbb{Z}$ such that $\Lambda^\#(i)=\Lambda(i+d)$, for all $i\in \mathbb{Z}$. 
For any integer $i$, let $a_i(\Lambda)$ be the set defined by
$$a_i(\Lambda):=\{T\in \End_F(W)\ |\ T\Lambda(j)\subset 
\Lambda(j+i)\ \forall\ j\in \mathbb{Z}\}.$$ Let $U_0(\Lambda)$ be the
set of units in $a_0(\Lambda)$. Let $U_i(\Lambda)$ be the group
$\id_V+a_i(\Lambda)$, for any $i> 0$.  Given any self-dual lattice
chain $\mathcal{L}$, there exists a standard basis $B$, called {\it
  a splitting} of $\Lambda$, such that for any
$i\in \mathbb{Z}$:
\begin{equation}\label{splitting}
\Lambda(i)=\mathfrak{p}_F^{a_{n}+i}w_n\oplus\mathfrak{p}_F^{a_{(n-1)}+i}w_{n-1}
\oplus\cdots\oplus \mathfrak{p}_F^{a_{(-n+1)}+i}w_{-n+1}\oplus 
\mathfrak{p}_F^{a_{(-n)}+i}w_{-n}.
\end{equation} 

Given any hyperspecial maximal compact subgroup $K$ of $G$, there
exists a self-dual lattice chain $\Lambda$ such that $K$ is equal to
$G\cap U_0(\Lambda)$. Note that $e(\Lambda)=1$. Let $K(m)$ be the
group $U_m(\Lambda)\cap G$, for $m\geq 1$. The group $K(m)$ is the
principal congruence subgroup of level $m$. The group $K(m)$ is a
normal subgroup of $K$, for $m\geq 1$. Let $B$ be a standard basis
such that $B$ is a splitting of $\Lambda$. Let ${\bf T}$ be the
maximal split torus of ${\bf G}$ associated to the standard basis
$B$. The group $K\cap T$ is the maximal compact subgroup of $T$.
Let $\mathcal{L}$ be the lattice
\begin{equation}\label{standard_lattice}
\mathcal{L}:=\Lambda(0)=
\mathfrak{p}_F^{a_{n}}w_n\oplus\mathfrak{p}_F^{a_{n-1}}w_{n-1}
\oplus\cdots\oplus \mathfrak{p}_F^{a_{-n+1}}w_{-n+1}\oplus 
\mathfrak{p}_F^{a_{-n}}w_{-n}.
\end{equation}
The lattice $\mathcal{L}$ is determined by the set of integers
$\{a_i:-n\leq i\leq n\}$. Let $L_0$ be the ideal generated by the set
$\{q(w_1,w_2): w_1, w_2\in \mathcal{L}\}$ in $\mathfrak{o}_F$. Let
$\bar{q}$ be the following bilinear form:
$$\bar{q}:\dfrac{\mathcal{L}}{\mathfrak{p}_F \mathcal{L}}
\times \dfrac{\mathcal{L}}{\mathfrak{p}_F \mathcal{L}} \rightarrow
\dfrac{L_0}{\mathfrak{p}_F L_0},\ \ q(w_1, w_2)\mapsto \overline{q(w_1
  ,w_2)}\ \forall \ w_1, w_2\in W,$$ where $\overline{q(w_1 ,w_2)}$ is
the image of $q(w_1 ,w_2)$ in $L_0/\mathfrak{p}_F L_0$. Since $K$ is
hyperspecial, the form $\bar{q}$ is non-degenerate (see
\cite[3.8.1]{tits_article}). We refer to the article \cite[Section
1.6]{lamaire_comparision} for these results.

Let ${\bf T}$ be any maximal split torus of ${\bf G}$, defined over
$F$, such that $K\cap T$ is the maximal compact subgroup of $T$. Let
$B$ be the standard basis of $W$ associated to the torus ${\bf
  T}$. There exists a self-dual lattice chain $\Lambda$ such that $B$
is a splitting of $\Lambda$, and $K$ is equal to $U_0(\Lambda)\cap G$.

Until the end of Section \ref{sec_auxiliary}, we fix a 
hyperspecial maximal compact subgroup $K$ of $G$. We
fix a self-dual lattice chain $\Lambda$ defining $K$. We fix a
standard basis
\begin{equation}\label{standard_basis_1}
B=(w_i:-n\leq i\leq n)
\end{equation}
such that $B$ is a splitting of $\Lambda$. We fix the set of integers
$\{a_i: -n\leq i\leq n\}$ as in \eqref{standard_lattice}. We have
a canonical homomorphism
\begin{equation}\label{mod_p_reduction}
\pi_1:K\rightarrow K/K(1)\simeq G(\mathcal{L}\otimes k_F, \bar{q}).
\end{equation}

Let $I$ be a sequence of positive integers
\begin{equation}\label{num_inv}
  n\geq n_1\geq n_2\geq \cdots\geq n_r\geq 1.
\end{equation} Consider the sets
$$S_i^{\pm}:=\{w_{\pm n}, w_{\pm(n-1)}, \cdots, w_{\pm(n_i)}\},$$
for $1\leq i\leq r$.  Let $W^{\pm}_i$ be the subspace of $W$ spanned
by the set $S_i^{\pm}$. We denote by $V^\pm_i$, the space spanned by
the set $S^\pm_{i+1}\backslash S^\pm_i$, for $i\leq r$. Let $V_{r+1}$
be the space $(W_{r}^+\oplus W_{r}^-)^\perp$. Let $\mathcal{F}_I$ be
the flag
\begin{equation}\label{flags} 
  W_1^+\subset W_2^+\subset \cdots \subset W_r^+.
\end{equation}
Let $P_I$ be the $G$-stabiliser of the flag $\mathcal{F}_I$. Let $M_I$
be the $G$-stabiliser of the decomposition
$$V_1^+\oplus\cdots \oplus  V_r^+\oplus
V_{r+1}\oplus V_r^-\oplus \cdots \oplus V_1^-.$$ The group $P_I$ is
the group of $F$-points of an $F$-parabolic subgroup of ${\bf G}$. Let
$U_I$ be the unipotent radical of $P_I$. We have $P_I=M_I\ltimes
U_I$. We denote by $\bar{U}_I$, the unipotent radical of the opposite
parabolic subgroup of $P_I$ with respect to the group $M_I$.

Assume that $G$ is a symplectic or special orthogonal group of odd
dimension. In this case, the group of $F$-points of any $F$-parabolic
subgroup of ${\bf G}$ is $G$-conjugate to $P_I$, for some sequence $I$
as in \eqref{num_inv}. The subgroups $P_I$ are called {\it standard
  parabolic subgroups}. The group $M_I$ will be called as a {\it
  standard Levi subgroup} of $P_I$.

Assume that $G$ is special orthogonal group of even dimension. In this
case, there are two orbits of maximal totally isotropic subspaces of
$W$. The representatives for these orbits are given by the spaces
\begin{eqnarray}
W^+=&Fw_n\oplus Fw_{n-1}\oplus\cdots\oplus Fw_{1}\\\label{flags_2}
(W^+)'=&Fw_n\oplus Fw_{n-1}\oplus\cdots\oplus Fw_{2}\oplus Fw_{-1}.
\end{eqnarray}
Let $\mathcal{F}_I'$ be a flag defined as in \eqref{flags}, except for
replacing $w_1$ with $w_{-1}$.  Let $P_{I}'$ and $M_{I}'$ be parabolic
subgroups, and Levi subgroups respectively, defined similarly as above
for the flag $\mathcal{F}_I'$. The group of $F$-points of an
$F$-parabolic subgroup of ${\bf G}$ is $G$ conjugate to at least one
of the groups $P_I$ or $P_{I}'$ for some sequence
$(n_1, n_2,\dots,n_r)$ as in \eqref{num_inv}. The parabolic subgroups
$P_I'$ and $P_I$ are called the {\it standard parabolic subgroups}. The
Levi factors $M_I$ and $M_I'$, for $P_I$ and $P_I'$ respectively, are
be called the {\it standard Levi subgroups}.
\begin{remark}
  There exist sequences $I$ such that $P_I$ and $P_I'$ are
  $G$-conjugate. Hence, for even special orthogonal groups these
  groups $P_I$ and $P_I'$ are not a parametrisation. Nevertheless, any
  parabolic subgroup of $G$ is conjugate to at least one such group.
\end{remark}

Let $P$ be a standard parabolic subgroup and $M$ be a standard Levi
factor of $P$. Let $U$ be the unipotent radical of
  $P$, and let $\bar{U}$ be the unipotent radical of the opposite
  parabolic subgroup, $\bar{P}$, of $P$ with respect to $M$. Let
$P(m)$ be the following compact open group of $G$:
$$P(m)=K(m)(P\cap K).$$
Note that the group $P(1)$ is a parahoric subgroup of $G$.  The group
$P(m)$ has an Iwahori decomposition with respect to the pair $(P,
M)$. The group $K/K(1)$ can be identified with $k_F$-points of the
connected component of the isometry subgroup associated to the pair
$(\mathcal{L}\otimes_{\mathfrak{o}_F} k_F, \bar{q})$; let $\pi_1$ be
the homomorphism as in \eqref{mod_p_reduction}.  Let $P(k_F)$ be the
image of $P(1)$ under $\pi_1$. $P(k_F)$ is a parabolic subgroup of
$K/K(1)$. The group $M(k_F)=\pi_1(K\cap M)$ is a Levi factor of
$P(k_F)$.

We identify $M$ with the group
$$G_1\times G_2\times \cdots\times G_r\times G_{r+1},$$
where $G_i={\rm GL}(V_i)$, for $1\leq i\leq r$, and $G_{r+1}$ is the
group of $F$-points of the connected component of the isometry group
associated to a non-singular subspace $(V_{r+1}, q)$ of $(W,q)$. Any
cuspidal representation $\sigma_M$ of $M$ is isomorphic to
$$\sigma_1\boxtimes\cdots\boxtimes\sigma_r\boxtimes \sigma_{r+1},$$
where $\sigma_i$ is a cuspidal representation of $G_i$, for $1\leq
i\leq r+1$.  Any inertial class $\mathfrak{s}$ of $G$ is equal to $[M,
\sigma_M]$.

Let $K_{M}$ be the group $M\cap K$. Note that $K_{M}$ is a
hyperspecial maximal compact subgroup of $M$.  Let
$\gamma_M$ be a cuspidal representation of $M(k_F)$. Let $\tau_M$ be
a representation of $K_{M}$, obtained as the inflation of $\gamma_M$
via the map
$$\pi_1:K_{M}=M\cap K \rightarrow M(k_F).$$ 
Note that $\tau_M$ extends as a representation of $P(1)$ via inflation
from the map
$$\tilde{\pi}_1: P(1)\xrightarrow{\pi_1} P(k_F)\rightarrow M(k_F).$$
Let $\sigma_M$ be a cuspidal representation of $M$ containing the pair
$(K_{M}, \tau_M)$.
\begin{lemma}
  Let $\mathfrak{s}$ be the inertial class $[M, \sigma_M]_G$. The pair
  $(P(1), \tau_M)$ is an $\mathfrak{s}$-type in the sense of Bushnell
  and Kutzko.
\end{lemma}
\begin{proof}
 This is essentially proved in \cite [Theorem
  4.9]{level_zero_G-types}. However, we have to show that the group
  $P(1)$ coincides with the full normaliser of the facet corresponding
  to the parahoric subgroup $P(1)$; which is denoted by $\hat{P}$ in
  \cite{level_zero_G-types}.  First, we have $P(1)\subseteq
  \hat{P}$. From the Iwahori decomposition of $\hat{P}$ with respect
  to $(P, M)$, we get that
  $$\hat{P}=(\hat{P}\cap U)(\hat{P}\cap M)(\hat{P}\cap \bar{U}).$$
  Since the groups $\hat{P}\cap U$ and $\hat{P}\cap \bar{U}$ are
  pro-$p$ groups, they are contained in $P(1)$. Since $K_M=P(1)\cap M$
  is a hyperspecial maximal compact subgroup, the group $ P(1)\cap M$
  is equal to $\hat{P}\cap M$.  This shows that $\hat{P}=P(1)$. 
\end{proof} In this article, we classify $[M, \sigma_M]_G$-typical
representation of $K$. In particular, we show that the
$[M, \sigma_M]_G$-typical representations of $K$ are exactly the
subrepresentations of $\ind_{P(1)}^K\tau_M$.
\section{The first reduction}
We begin with a few preliminary results.  We will make a mild
modification to the uniqueness result of typical representations
proved for depth-zero inertial classes of ${\rm GL}_n(F)$. The following
lemmas are essentially proved by Pa{\v{s}}k{\=u}nas in
\cite{Paskunas-uniqueness}, but not stated in the form we need.
\begin{lemma}\label{ref_disct}
  Let $G$ be the group of $k_F$-points of a connected reductive group
  over $k_F$. Let $H$ be a subgroup of $G$. Assume that there exists a
  proper parabolic subgroup $P$ of $G$, with unipotent radical $U$
  such that $H\cap U=\{\id\}$. Let $\tau$ be an irreducible
  representation of $G$. For any irreducible subrepresentation
  $\xi$ of $\res_H\tau$, there exists an irreducible non-cuspidal
  $G$-representation $\tau'$ such that $\xi$ occurs as a
  subrepresentation of $\res_H\tau'$.
\end{lemma}
\begin{proof}
Using Mackey decomposition, we observe that the space
$$\ho_U(\ind_H^G\xi, \id)$$ is non-trivial. Therefore, there exists an
irreducible non-cuspidal $G$-subrepresentation $\tau'$ of
$\ind_H^G\xi$. Frobenius reciprocity implies that $\xi$ occurs in the
irreducible non-cuspidal representation $\tau'$ of $G$. 
\end{proof}

 For simplicity until the end of
Lemmas \ref{unicity_gln} and \ref{unicity_gln_noncusp}, we denote the
group ${\rm GL}_n(F)$ by $G_n$ and the group
${\rm GL}_n(\mathfrak{o}_F)$ by $K_n$.
 \begin{lemma}\label{unicity_gln}
   Let $n>1$, and let $\mathfrak{s}=[G_n, \sigma]_{G_n}$ be a
   depth-zero inertial class.  The representation $\res_{K_n}\sigma$
   admits a decomposition:
   $$\res_{K_n}\sigma=\tau\oplus \tau'$$ such that 
   $\tau$ is an $\mathfrak{s}$-typical representation of $K_n$, and any
   irreducible $K_n$-subrepresentation $\xi$ of $\tau'$ occurs in
   $\res_{K_n}\pi_\xi$, for some irreducible non-cuspidal
   representation $\pi_\xi$ of $G$.
  \end{lemma}
  \begin{proof}
    The representation $\sigma$ is an unramified twist of the
    representation $\ind_{F^\times {K_n}}^{G_n}\tau$, where $\tau$ is
    a representation of $F^\times K_n$ such that: $\res_{K_n}\tau$ is
    obtained by inflation of a cuspidal representation of
    ${\rm GL}_n(k_F)$, and $\varpi_F$ acts trivially on $\tau$. Using
    Cartan decomposition for the group $G_n$, the representatives for
    the double cosets $F^\times {K_n}\backslash G_n/{K_n}$ are given
    by the elements of the form
    $\diag(\varpi_F^{i_1}, \dots, \varpi_F^{i_n})$, where
    $i_1\geq \cdots\geq i_n\geq 0$.  Now
  $$\res_{K_n}\sigma\cong\bigoplus_{t\in {K_n}\backslash {\rm GL}_n(F)/{K_n}}
  \ind_{{K_n}\cap t{K_n}t^{-1}}^{K_n}\tau.$$ Assume $t\neq \id$. Let
  $H$ be the image of the group ${K_n}\cap t{K_n}t^{-1}$ under the
  reduction map $\pi_1:{K_n}\rightarrow {\rm GL}_n(k_F)$. The group
  $H$ is contained in a proper parabolic subgroup $Q$ of
  ${\rm GL}_n(k_F)$.

Let $U$ be the unipotent radical of an opposite
  parabolic subgroup of $Q$. Note that $H\cap U$ is the trivial
  group. Let $\xi$ be an irreducible $H$-subrepresentation of
  $\tau$. Using Lemma \ref{ref_disct}, we get that $\xi$ occurs as a
  subrepresentation of $\res_H \gamma$, where $\gamma$ is a non-cuspidal
  irreducible representation of ${\rm GL}_n(k_F)$. This implies that
  any irreducible subrepresentation of
  $\res_{{K_n}\cap t{K_n}t^{-1}}\tau $ occurs as a subrepresentation
  of $\res_{{K_n}\cap t{K_n}t^{-1}}\tau'$ where $\tau'$ is the
  inflation of $\gamma$. This shows that any $K_n$-irreducible
  subrepresentation  of $\ind_{{K_n}\cap t{K_n}t^{-1}}^{K_n}\tau$
  occurs in $\ind_{{K_n}\cap t{K_n}t^{-1}}^{K_n}\tau'$, for some
  $\tau'$ as above. 

  The representation $\ind_{{K_n}\cap t{K_n}t^{-1}}^{K_n}\tau'$ is a
  subrepresentation of $\res_{K_n}\ind_{K_n}^G\tau'$. Let $Q(1)$ be a
  subgroup of $K_n$, obtained as the inverse image of $Q$ via the map
  $\pi_1:K_n\rightarrow \g{n}{k_F}$. Let $N$ be a Levi factor of $Q$.
  The representation $\gamma$ is a subrepresentation of
  $i_Q^{\g{n}{k_F}}\gamma_N$, where $\gamma_N$ is a cuspidal
  representation of $N$. Let $\tau_N$ be the representation of $Q(1)$
  obtained by inflation of $\gamma_N$ via the map $\pi_1:Q(1)\rightarrow Q$.
  The representation $\ind_{K_n}^G\tau'$ is a subrepresentation of
  $\ind_{Q(1)}^G\tau_N$. Any irreducible $G$-subquotient of
  $\ind_{Q(1)}^G\tau_N$ is a non-cuspidal representation (see
  \cite[chapter 8]{Orrangebook}). This shows that irreducible
  subrepresentations of $\ind_{{K_n}\cap t{K_n}t^{-1}}^{K_n}\tau'$
  occur in the restriction to $K_n$ of a non-cuspidal representation
  of $G$.
  \end{proof}
\begin{lemma}\label{unicity_gln_noncusp}
  Let $\mathfrak{s}=[M, \sigma]_{G_n}$ be a depth-zero non-cuspidal
  inertial class. Let $P$ be a parabolic subgroup of
  $G$ such that $M$ is a Levi factor of $P$.  The representation
  $\res_{K_n}i_P^{G_n}\sigma$ admits a
  decomposition $$\res_{K_n}i_P^{G_n}\sigma=\tau\oplus \tau'$$ such
  that  any irreducible $K_n$-subrepresentation of $\tau$ is
  $\mathfrak{s}$-typical, and any irreducible $K_n$-subrepresentation
  of $\tau'$ is atypical. Moreover, any irreducible
  $K_n$-subrepresentation of $\tau'$ occurs as a subrepresentation of
  $\res_{K_n}i_{R}^{G_n}\sigma_1$ such that $P$ and $R$ are not
  associate parabolic subgroups.
\end{lemma}
\begin{proof}
  The first part of the lemma is proved in \cite[Theorem
  3.2]{level_zero_gl_n_types}. The last assertion follows from the
  proof of the result \cite[Theorem 3.2]{level_zero_gl_n_types}.  Note
  that there are no assumptions on $q_F$ in the proof of this lemma.
\end{proof}

Let $K$ be any hyperspecial maximal compact subgroup of $G$. We need
the uniqueness of $\mathfrak{s}$-typical representations of $K$ for
the inertial class $[G, \sigma]$, where $\sigma$
contains a depth-zero type of the form $(K, \lambda)$. We only give a
sketch of the following standard lemma for the completeness of the
exposition. This result is generalised by Latham for arbitrary maximal
compact subgroups and depth-zero cuspidal Bernstein components of an
wide class of reductive groups $G$ (see \cite{latham_2}).
\begin{lemma}\label{unicity_classical_cusps}
  The $K$-representation $\lambda$, is the
  unique $[G, \sigma]_G$-typical representation contained in
  $\sigma$. 
\end{lemma} 
\begin{proof}
The representation $\sigma$ is isomorphic to $\ind_K^G\lambda$. Now 
$$\res_K\ind_{K}^G\lambda\simeq \bigoplus_{g\in K\backslash G/ K}
\ind_{K^g\cap K}^K\lambda^g.$$ Assume that $g\notin K$.  Observe that
Cartan decomposition for $K\backslash G/K$ gives a representative
$t\in KgK$ such that $K^{t^{-1}}\cap K\subset P(1)$, for some proper
standard parabolic subgroup $P$ of $G$.  Using Lemma \ref{ref_disct},
we get that any irreducible subrepresentation $\xi$ of
$$\res_{K^{t^{-1}}\cap K}\lambda$$
occurs as a subrepresentation of $\res_{K^{t^{-1}}\cap
  K}\ind_{R(1)}^K\tau'$, where $\tau'$ is the inflation of a cuspidal
representation $\gamma$ of $L(k_F)$, the standard Levi factor of
$R(k_F)$, via the map
$$R(1)\rightarrow R(k_F)\rightarrow L(k_F).$$

Hence, any irreducible representation of $\ind_{K^g\cap K}^K\lambda^g$
occurs as a subrepresentation of $$\res_K\ind_{R(1)}^G\tau'.$$
The pair $(R(1), \tau')$ is a type for the Bernstein component $[L,
\sigma_L]$, where $\sigma_L$ is any cuspidal representation of $L$
containing the type $(K\cap L, \tau')$. Now any irreducible
$G$-subquotients of $\ind_{R(1)}^G\tau'$ are non-cuspidal. Hence the
irreducible subrepresentations of $\ind_{K^g\cap K}^K\lambda^g$ are
atypical. 
\end{proof}

Consider a standard parabolic subgroup $P$ with the standard Levi factor
$M$ isomorphic to
$$G_1\times G_2\times \cdots\times G_{r+1},$$
where $G_i$ is the group of $F$-points of a general linear group over
$F$, for $i\leq r$, and $G_{r+1}$ is the group of $F$-points of the
connected component of the isometry subgroup of a non-singular
subspace $(W', q)$ of $(W, q)$.  The factor $G_{r+1}$ is assumed to be
trivial if $M$ is contained in a maximal parabolic subgroup fixing a
maximal totally isotropic flag.  Let
$\mathfrak{t}_i=[M_i, \sigma_i]_{G_i}$ be an inertial class of $G_i$,
for $i\leq r$ and $\mathfrak{t}_{r+1}=[G_{r+1}, \sigma_{r+1}]$ be a
cuspidal inertial class of $G_{r+1}$.

We assume that $\mathfrak{t}_i$ is a depth-zero inertial class of
$G_i$, for $1\leq i\leq r$. We assume that $\sigma_{r+1}$ contains a
depth-zero type $(K\cap G_{r+1}, \lambda)$. Let $P_i$ be an
$F$-parabolic subgroup of $G_i$ with $M_i$ as a Levi factor, and let
\begin{equation}\label{basic_decomp}
\res_{K\cap G_i}i_{P_i}^{G_i}\sigma_i=\tau_i\oplus \tau_i'
\end{equation}
such that: any $K\cap G_i$-irreducible subrepresentation of $\tau_i'$
is atypical, $\tau_i\neq 0$, and any $K\cap G_i$-subrepresentations
of $\tau_i$ is $\mathfrak{t}_i$-typical. Such a decomposition is
possible by Lemmas \ref{unicity_gln} and \ref{unicity_gln_noncusp}, for
$i\leq r$, and for $G_{r+1}$ from the Lemma
\ref{unicity_classical_cusps}.

Let $\mathfrak{s}$ be the inertial class $[L, \sigma_L]_G$, where
$L\subset M$, is a standard Levi factor of a standard parabolic
subgroup such that
$$L\simeq M_1\times\cdots\times M_r\times G_{r+1},$$ 
and $\sigma_L$ is isomorphic to
$\sigma_1\boxtimes\cdots\boxtimes\sigma_r\boxtimes\sigma_{r+1}$.  We
denote by $\tau_{M}$ the $K \cap M$-representation
$$\tau_1\boxtimes\tau_2\boxtimes\cdots\boxtimes\tau_{r+1}.$$
Let $R$ be a standard parabolic subgroup such that $L$ is the standard
Levi factor of $R$.  Let $\tau'_{M}$ be the representation
$\ind_{R\cap M}^M\sigma_L/\tau_M$.  With these notations, we have the
following preliminary classification of $\mathfrak{s}$-typical
representations of $K$. 
\begin{lemma}\label{conley_lemma}
  Let $\mathfrak{s}$ be the inertial class $[L, \sigma_L]_G$. Any
  $\mathfrak{s}$-typical representation $\tau$ of $K$ occurs as a
  subrepresentation of $\ind_{K\cap P}^K\tau_M$.
\end{lemma}
\begin{proof}
  The representation $\ind_{K}^{G}\tau$ is finitely generated and
  hence has an irreducible quotient $\pi$. From Frobenius reciprocity,
  the representation $\pi$ occurs as a subquotient of
  $i_R^G(\sigma_L\otimes\chi)$, where $R$ is a standard parabolic
  subgroup $G$ with Levi factor $L$, and $\chi$ is some
  unramified character of $L$. 

  Let $\tilde{\sigma}_M$ be the representation
  $i_{R\cap M}^{M}\sigma_L$. Then $\tau$ occurs as a
    subrepresentation of $\res_Ki_R^G\sigma_L$, and we have
  restriction
  $$\res_Ki_R^G\sigma_L=\ind_{P\cap K}^K(\res_{K\cap M}\tilde{\sigma}_M)
  =\ind_{P\cap K}^K\tau_{M}\oplus
\ind_{P\cap K}^K\tau_M'.$$ The Levi subgroup $M$ is isomorphic to
$G_1\times G_2\times\cdots\times G_r\times G_{r+1}$. We identify
$\tilde{\sigma}_M$ with the representation
$\tilde{\sigma}_1\boxtimes\tilde{\sigma}_2
\boxtimes\dots\boxtimes\tilde{\sigma}_r\boxtimes\tilde{\sigma}_{r+1}$,
where $\tilde{\sigma}_i$ is the representation
$i_{P_i}^{G_i}(\sigma_i\otimes\chi_i)$. Here $P_i$ is the parabolic
subgroup $R\cap G_i$ of $G_i$, containing $M_i$ as a Levi factor and
$\chi_i=\res_{M_i}\chi$ is an unramified character of $M_i$ for all
$1\leq i\leq r+1$.

Let $$\res_{K\cap G_i}\tilde{\sigma_i}=\oplus_j\xi_i^j,$$ where
$\xi_i^0=\tau_i$ as defined in the decomposition of
$\res_{K\cap G_i}\tilde{\sigma}_i$ in \eqref{basic_decomp}, and for
$j>0$ the representation $\xi_i^j$ is an irreducible subrepresentation
of $\tau_i'$ in \eqref{basic_decomp}. Now the representation $\tau_M$
is isomorphic to
$\xi_1^0\boxtimes\cdots\boxtimes\xi_r^0\boxtimes\xi_{r+1}^0$. Similarly
define the representation $\tau'_M$ as the representation
$$\bigoplus_{(i_1, i_2, \cdots,
  i_{r+1})\neq
  0}\xi^{i_1}_1\boxtimes\xi^{i_2}_2\boxtimes\cdots\boxtimes
\xi^{i_{r+1}}_{r+1}.
$$ 
We denote by $\xi_I$, the summand corresponding to the 
tuple $I=(i_1, i_2, \cdots,i_{r+1})$. Let $I$ be the non-zero tuple
$(i_1, i_2, \cdots,i_{r+1})$ and fix $1\leq j\leq r+1$ such that $i_j\neq
0$.  
Now $\xi_j^{i_j}$ is atypical and hence occurs in 
$$\res_{K\cap G_j}i_{R_j'}^{G_j}\gamma_j$$
where $R_j'$ is a parabolic subgroup of $G_j$, with a
Levi factor $M_j'$, $\gamma_j$ is a cuspidal representation of $M_j'$
such that $[M_j', \gamma_j]$ is not equal to $[M_j, \sigma_j]$.

Let $L'$ be the Levi subgroup $M_1\times M_2\times\cdots\times
M_{j-1}\times M_j'\times \cdots\times G_{r+1}$ and $\sigma'_{L'}$ be
the cuspidal representation
$\sigma_1\boxtimes\cdots\boxtimes\sigma_{j-1}
\boxtimes\gamma_j\boxtimes\cdots\boxtimes \sigma_{r+1}$. Let $R'$ be
any parabolic subgroup such that $L'$ is a Levi factor of $R'$.  Note
that
$$\ind_{K\cap P}^K\xi_I\subset \res_Ki_{R'}^{G}\sigma'_{L'}.$$
Now the cuspidal support of $i_{R'}^{G}\sigma'_{L'}$ is given by
$[L', \sigma'_{L'}]$. If $j<r+1$ then using Lemmas \ref{unicity_gln}
and \ref{unicity_gln_noncusp}, we know that $M_j$ and $M_j'$ are not
conjugate in $G_j$. This shows that $L$ and $L'$ are not conjugate in
$G$.  Hence the inertial class $[L', \sigma'_{L'}]$ is not equal to
$[L, \sigma_L]$.  Assume that $j=r+1$. In this case, Lemma
\ref{unicity_classical_cusps} shows that $L'$ is a proper Levi
subgroup of $L$.  Hence the pairs $(L, \sigma_L)$ and
$(L', \sigma'_{L'})$ represent two distinct inertial classes. This
shows that any irreducible subrepresentation of
$\ind_{K\cap P_I}^K\xi_I$ is atypical.
\end{proof}
\section{Decomposition of an auxiliary representation }
\label{sec_auxiliary}
Let $P$ be any standard parabolic subgroup of $G$. Let $U$ be the
unipotent radical of $P$. Let $M$ be the standard Levi subgroup of
$P$. Let $\bar{P}$ be the opposite parabolic subgroup of $P$ with
respect to $M$. Let $\bar{U}$ be the unipotent radical of $\bar{P}$.
Let $\mathfrak{s}=[M, \sigma_M]$ be a depth-zero Bernstein component
such that $\sigma_M$ contains a type $(K_{M}, \tau_M)$, where $\tau_M$
is the inflation of a cuspidal representation $\gamma_M$ of $M(k_F)$.

Let $m\geq 1$ be any positive integer. Recall that $P(m)$ is defined
as the group $(P\cap K)K(m)$. The group $P(m)$ has Iwahori decomposition
with respect to the pair $(P, M)$. Moreover,
$$P(m)\cap M=K\cap M~{\rm and}~
P(m)\cap U=U\cap K.$$ The representation $\tau_M$ extends as a
representation of $P(m)$ via inflation from the map
$\pi_1:P(1)\rightarrow P(k_F)$ defined in \eqref{mod_p_reduction}. The groups
$U\cap P(m)$ and $\bar{U}\cap P(m)$ are contained in the kernel of
this inflation.  Note that
\[\bigcap_{m\geq 1}P(m)=P\cap K.\] We obtain 
$$\ind_{K\cap P}^K\tau_M=\bigcup_{m\geq 1}\ind_{P(m)}^K\tau_M.$$
We will show that the irreducible subrepresentations of the quotient
$$\ind_{P(m+1)}^K\tau_M/(\ind_{P(m)}^K \tau_M)$$ 
are atypical.

Given any irreducible representation $\tau$ of $M(k_F)$, we consider
$\tau$ first as a representation of $P(k_F)$ via inflation. Then
$\tau$ is considered as a representation of $P(1)$ via inflation from
the map $\pi_1:P(1)\rightarrow P(k_F)$ in
\eqref{mod_p_reduction}. There exists a standard parabolic subgroup
$R\subset P$ in $G$, containing $L$ as its standard Levi factor, such
that: $L\subset M$, and $\tau$ is a subrepresentation of
$$\ind_{R(k_F)\cap M(k_F)}^{M(k_F)}\tau',$$ where $\tau'$ is a cuspidal
representation of $L(k_F)$.  If
$${\rm
  Hom}_{P(1)}(\tau, \pi)\neq 0,$$ for some irreducible smooth
representation $\pi$ of $G$, then the representation $\tau'$ of $R(1)$
occurs in $\pi$. The cuspidal support of the representation $\pi$ is
$[L, \sigma_L]$, where $\sigma_L$ is a cuspidal representation of $L$
containing the pair $(K_L,\tau')$. We call the component
$[L, \sigma_L]_G$ as the {\it inertial class associated to the pair
  $(P(1), \tau)$}.
 
For the purpose of inductive arguments it is useful to introduce some
more classes of compact open subgroups and prove some basic properties
of these groups.  Let $I$ be a sequence of integers 
$$n\geq n_1\geq \cdots\geq n_r\geq 1.$$
Let $I_1$ be the sequence of integers as above consisting of a single
integer $n_r$.  Let $\mathcal{F}_I$ be the flag
$W^+_1\subset \cdots\subset W^+_r$ of totally isotropic subspaces of
$W$, as defined in \eqref{flags}, corresponding to $I$ (or possibly
the flag defined for \eqref{flags_2}, if $G$ is isomorphic to special
orthogonal subgroup ${\rm SO}_{2n}(F)$). Let $P$ be the standard
parabolic subgroup fixing the flag $\mathcal{F}_I$. Let
$\mathcal{F}_{I_1}$ be the flag $W^+_r$ (or possibly the space
$(W_r^+)'$ if $G$ is isomorphic to ${\rm SO}_{2n}(F)$). The standard
parabolic subgroup $P_{1}$ fixing the flag $\mathcal{F}_{I_1}$ is the
maximal proper parabolic subgroup containing the parabolic subgroup $P$. Let
$M_1$ be the standard Levi factor of $P_1$. Let $U_1$ be the unipotent
radical of $P$. Let $\bar{P}_1$ be the opposite parabolic subgroup of
$P_1$ with respect to $M_1$. Let $\bar{U}_1$ be the unipotent radical
of $P_1$.

Let $1\leq i\leq r$ be any positive integer. Let $\bar{V}_i^\pm$ be
the subspace $\mathcal{L}\otimes k_F$ spanned by set of vectors
$\{\varpi_F^{a_i}w_i\otimes 1\ |\ w_i\in S_i^\pm\}$. Let
$\bar{V}_{r+1}$ be the space $(\bar{W}_r^+\oplus
\bar{W}_r^-)^\perp$. Let $\bar{W}_i$ be the totally isotropic space
$$\bar{V}_1^+\oplus \bar{V}_2^+\oplus\cdots\oplus \bar{V}_i^+.$$
The parabolic subgroup $P(k_F)$ is the $G(\mathcal{L}\otimes k_F,
\bar{q})$-stabilizer of the flag
$$\bar{W}_1^+\subset \bar{W}_2^+\subset \cdots\subset \bar{W}_r^+.$$
The group $M(k_F)$ is the  $G(\mathcal{L}\otimes k_F,
\bar{q})$-stabilizer of the decomposition 
$$\bar{V}_1^+\oplus \bar{V}_2^+\oplus\cdots\oplus \bar{V}_r^+\oplus
\bar{V}_{r+1}
\oplus\bar{V}_r^-\oplus \bar{V}_{r-1}^-\oplus\cdots\oplus
\bar{V}_1^-.$$
Moreover, the group $P_1(k_F)$ is the $G(\mathcal{L}\otimes k_F,
\bar{q})$-stabilizer of the space $\bar{W}_r^+$, and 
$M_1(k_F)$ is the $G(\mathcal{L}\otimes k_F,
\bar{q})$-stabilizer of the decomposition 
$$W_r^+\oplus V_{r+1}\oplus W_r^-.$$

Let $m$ be a positive integer. We introduce a compact open subgroup
$P(1, m)\subseteq P(1)$, which helps in inductive arguments.  We set
$$P(1,m) =K(m)(P(1)\cap P_{1}).$$ 
Using Iwahori decomposition of the group $K(m)$,
we get that the group $P(1,m)$ admits an Iwahori decomposition with
respect to the pair $(P_{1}, M_{1})$. Let $U_1$ be the unipotent
radical of $P_1$ and $\bar{U}_1$ be the unipotent radical of the
opposite parabolic subgroup of $P_1$ with respect to $M_1$. Using the
Iwahori decomposition of $P(1)$ with respect to the pair $(P_1, M_1)$,
we get that 
$$P(1)=(P(1)\cap \bar{U}_1)(P(1)\cap P_1).$$
Now, the group $P(1)\cap \bar{U}$ is contained in $K(1)$. Hence, we
have $P(1, 1)=P(1)$.  One of the main ingredient in classification of
typical representations is the description of the induced
representation
 $$\ind_{P_1(1,m+1)}^{P_1(1,m)}\id.$$
 
Since the unipotent radical of $P_{1}$ is not necessarily abelian, it
 is useful to introduce another family of compact subgroups $R(m)$
 such that
$$P(1,m+1) \subset R(m) \subset P(1,m).$$ 
With respect to the basis 
\begin{equation}
(\varpi_F^{a_n}w_n, \varpi_F^{a_{n-1}}w_{n-1}, \dots, \varpi_F^{a_{-n+1}}w_{-n+1},
  \varpi_F^{a_{-n}}w_{-n}),
\end{equation} we identify the group $K$ as a
subgroup of ${\rm GL}_N(\mathfrak{o}_F)$ and
$P$ as a subgroup of invertible upper block matrices. With this identification,
let $R(m)$ be the compact open subgroup of $P(1,m)$ consisting of
matrices of the form:
$$\begin{pmatrix}
  *&*&*&*&*\\
  *&*&*&*&*\\
  *&*&*&*&*\\
  *&*&*&*&*\\
  Z&*&*&*&*
\end{pmatrix}$$ where entries of the matrix $Z$ belong to
$M_{n_r\times n_r}(\mathfrak{p}_F^{m+1})$. Since $m\geq 1$,
the group $R(m)$ is well defined. Let $\mathfrak{n}_1$ be the Lie
algebra of $\bar{U}_1(k_F)$.  Now, with respect to the basis
\begin{equation}
  (\varpi_F^{a_n}w_n\otimes 1, \varpi_F^{a_{n-1}}w_{n-1}\otimes 1, 
  \dots, \varpi_F^{a_{-n+1}}w_{-n+1}\otimes 1,
  \varpi_F^{a_{-n}}w_{-n}\otimes 1),
\end{equation} of
$\mathcal{L}\otimes k_F$, let $\bar{\mathfrak{n}}^1_{1}$ and
$\bar{\mathfrak{n}}_{1}^2$ be the space of matrices in $\mathfrak{n}_1$
of the form
$$\begin{pmatrix}
  0&0&0&0&0\\
  X&0&0&0&0\\
  a&0&0&0&0\\
  Y&0&0&0&0\\
  0&Y'&a'&X'&0
\end{pmatrix}\ \text{and}
\begin{pmatrix}
0&0&0&0&0\\
0&0&0&0&0\\
0&0&0&0&0\\
0&0&0&0&0\\
Z&0&0&0&0
\end{pmatrix}$$ respectively, where
$X, Y, (X')^{\text{tr}}, (Y')^{\text{tr}}\in M_{(n-n_r)\times
  n_r}(k_F)$, and $a, (a')^{\text{tr}}\in M_{1\times n_r}(k_F)$.  The
space $\mathfrak{n}_1$ is equal to
$\mathfrak{n}_1^1\oplus \mathfrak{n}_1^2$. Note that for symplectic
groups and even orthogonal groups, the $n+1$-th rows and columns are
assumed to be absent.

Now we want to decompose the representations
$$\ind^{P(1,m)}_{R(m)}\id~{\rm and}~\ind^{R(m)}_{P(1,m+1)}\id.$$ We
first consider two normal subgroups $K_1$ and $K_2$ of $P(1,m)$ and
$R(m)$ respectively, with the properties that $$K_1\cap R(m)\trianglelefteq 
K_1~ {\rm and}~ K_2\cap P(1,m)\trianglelefteq K_2.$$ The groups $K_1$
and $K_2$ are kernels of the quotient maps $$P(1,m)\rightarrow
M_{1}(k_F)~{\rm and}~R(m)\rightarrow M_{1}(k_F)$$ respectively. Since
$K_1$ and $K_2$ differ from $P(1,m)$ and $R(m)$ only by their 
intersections with Levi group $M_{1}$,
we get that $$K_1R(m)=P(1,m)~{\rm and}~K_2P(1,m+1)=R(m).$$
\begin{lemma}\label{normality_lemma}
The subgroup $K_1\cap R(m)$ is a normal subgroup of $K_1$ and
$K_2\cap P(1,m+1)$ is a normal subgroup of $K_2$. 
\end{lemma}
\begin{proof}
  The groups $K_1$ and $K_2$ satisfy Iwahori decomposition with
  respect to the pair $(P_1, M_1)$.  Observe that $$K_1\cap
  P_{1}=(K_1\cap R(m))\cap P_{1}~{\rm and}~K_2\cap P_{1}=(K_2\cap
  P(1, m+1))\cap P_{1}.$$ We need to check that $K_1\cap
  \bar{U}_{1}$ normalizes $K_1\cap R(m)$, and $K_2\cap \bar{U}_{1}$
  normalizes $K_2\cap P_I(1, m+1)$. We have $M_{1}\cap
  P(1,m)$-equivariant isomorphisms
$$\dfrac{K_1\cap \bar{U}_1}{(K_1\cap R(m))\cap \bar{U}_1}\simeq 
\bar{\mathfrak{n}}_{1}^1$$ and
$$\dfrac{K_2\cap \bar{U}_{1}}{(K_2\cap P_I(1,m+1))\cap \bar{U}_1}\simeq 
\bar{\mathfrak{n}}_{1}^2.$$ 
Since $K_1\cap M_{1}$ (respectively
$K_2\cap M_{1}$) acts trivially on $\bar{\mathfrak{n}}_{1}^1$
(respectively on $\bar{\mathfrak{n}}_{1}^2$), we get that
$u^{-}j(u^{-})^{-1}$ belongs to $K_1\cap R(m)$ (respectively $K_2\cap
P(1,m)$) for all $u^- \in K_i\cap \bar{U}_{1}$ and $j\in K_i\cap
M_{1}$ for $i\in \{1,2\}$.  

With this, we are left with showing that
$u^{-}u^+(u^{-})^{-1}$ belongs to $K_1\cap R(m)$ (respectively
$K_2\cap P(1,m)$) for all $u^{-}$ in $K_1\cap \bar{U}_{1}$
(respectively $K_2\cap \bar{U}_{1}$) and $u^{+}$ in $K_1\cap U_{1}$
(respectively $K_2\cap U_{1}$). {\bf We break the verification in two
cases when $W_r$ is maximal or non-maximal totally isotropic
subspace}. Because of dimension reason, we consider the symplectic and
even orthogonal cases first and then consider the odd orthogonal case.

For any block matrix $A$ in $M_{m\times n}(\mathfrak{o}_F)$, let
${\rm val}(A)$ be the least positive integer $k$ such that
$A\in M_{m\times n}(\mathfrak{p}_F^k)$.  Let $t$ be the dimension of
$W_r$. First, suppose $W_r$ is a maximal totally isotropic space,
i.e., $t=n$.  Consider the case where $G$ is either a symplectic or
even orthogonal group. In this case, we have
$R(m)=P(1,m+1)$. let $$\begin{pmatrix}
  I_n&0 \\
  X&I_n \\
\end{pmatrix}\in K_1\cap \bar{U}_1~
{\rm and}~\begin{pmatrix}
  I_n&A \\
  0&I_n \\
\end{pmatrix}\in K_1\cap U_1,$$ where $X\in{\rm M}_{n}(\mathfrak
p_F^{m+1})$ and $A\in{\rm M}_{n}(\mathfrak o_F)$.
We have 
$$\begin{pmatrix}
I_n&0 \\
X&I_n \\
\end{pmatrix}\begin{pmatrix}
I_n&A \\
0&I_n \\
\end{pmatrix}\begin{pmatrix}
I_n&0 \\
X&I_n \\
\end{pmatrix}^{-1} =\begin{pmatrix}
  I_n-AX&A \\
  -XAX&I_n+XA \\
\end{pmatrix}.$$ 
The lemma in this situation follows from the observation that
$XAX\in {\rm M}_n(\mathfrak{p}_F^{m+1})$. For odd orthogonal
groups, $$u^-=\begin{pmatrix}
I_n&0&0 \\
a&1&0 \\
X&a'&I_n \\
\end{pmatrix}
\ \text{and}\ 
u^+=\begin{pmatrix}
  I_n & b & Y \\
  0 & 1 & b' \\
  0 & 0 & I_n \\
\end{pmatrix},$$
where $a'$ and $b'$ are uniquely determined by $a$ and $b$
respectively. Now, the matrix $u^-u^+{(u^-)}^{-1}$ in its block matrix
form as above is equal to 
$$\begin{pmatrix}
\ast&\ast&\ast\\
a_1&\ast&\ast\\
X_1& a_1'&\ast
\end{pmatrix},$$
where 
\begin{align*}
&a_1=-aba-(ay+b')(X+a'a),\\
&X_1=X-(Xb+a')a-(XY+a'b'+1)(X+aa')\\
&a_1'=Xb-(XY+a'b')a'.
\end{align*}
Clearly, ${\rm val}(a_1)$, ${\rm val}(a_1')$ and ${\rm val}(X_1)$ are
greater than or equal to $m+1$. This shows that $u^-u^+{(u^-)}^{-1}\in
K_1\cap R(m)$ for similar reasons.

Now assume that $W_r$ is a non-maximal totally isotropic subspace of
$W$, i.e. $t<n$. We first consider the symplectic or even orthogonal
case. Let
$$u^-=\begin{pmatrix}
I_t&0&0&0 \\
A&I_{n-t}&0&0 \\
B&0&I_{n-t}&0 \\
C&B'&A'&I_t \\
\end{pmatrix}\in K_i\cap\bar{U}_1
\ \text{and}\ u^+=\begin{pmatrix}
  I_t & X & Y & Z \\
  0 & I_{n-t} & 0 & Y' \\
  0 & 0 & I_{n-t} & X' \\
  0 & 0 & 0 & I_t \\
\end{pmatrix}\in K_i\cap U_{1},$$
for $i=1,2$. Hence ${\rm val}_F\{A,B,C\}\geq m$. Here again,
 $A',B',X'$ and $Y'$ are uniquely determined
by $A,B,X,$ and $Y$ respectively. The matrix
$u^-u^+(u^-)^{-1}$ looks like
$$u^-u^+(u^-)^{-1}=\begin{pmatrix}
  *&*&*&* \\
  P&*&*&* \\
  Q&*&*&* \\
  R&Q'&P'&* \\
\end{pmatrix},$$
where
\begin{align}\label{relations}
P=&-AXA-AYB-AZC-Y'C, \nonumber\\
Q=&-BXA-BYB-BZC-X'C, \\
R=&-CXA-B'A-CYB-A'B-CZC-B'Y'C-A'X'C. \nonumber
\end{align}
Since ${\rm val}_F(R)\geq m+1$, it follows that $K_1\cap R(m)$ is
normal in $K_1$. The remaining case, i.e. $K_2\cap P(m+1)$ is normal
in $K_2$ is similar. Indeed, in this case ${\rm val}_F\{A,B\}\geq m$
and ${\rm val}_F(C)\geq m+1$. Hence normality follows from the fact that
${\rm val}_F\{P,Q\}\geq m+1$.

Now finally we consider the odd orthogonal case.  We have
$$u^-=\begin{pmatrix}
I_t&0&0&0&0 \\
A&I_{n-t}&0&0&0 \\
x&0&1&0&0 \\
B&0&0&I_{n-t}&0 \\
C&B'&x'&A'&I_t \\
\end{pmatrix}
\ \text{and}\ 
u^+=\begin{pmatrix}
  I_t & X & a & Y & Z \\
  0 & I_{n-t}&0 & 0 & Y' \\
0&0&1&0&a' \\  
0 & 0 &0& I_{n-t} & X' \\
  0 & 0&0 & 0 & I_t \\
\end{pmatrix},$$ where $x\in {\rm M}_{1,t}(\mathfrak p_F^{m+1})$. Let
$A_1$ denote the matrix $\begin{pmatrix}
  A \\
  x
\end{pmatrix}\in M_{n-t+1,t}(\mathfrak p_F^{m+1})$. Similarly, We define the
matrix $X_1$ to be $X_1=(X \ a)\in M_{t,n-t+1}(\mathfrak o_F)$. 
After redefining $B'$ and $Y'$ appropriately, we get 
$$u^-=\begin{pmatrix}
I_t&0&0&0 \\
A_1&I_{n-t+1}&0&0 \\
B&0&I_{n-t}&0 \\
C&B'&A'&I_t \\
\end{pmatrix}
\ \text{and}\ 
u^+=\begin{pmatrix}
  I_t & X_1  & Y & Z \\
  0 & I_{n-t+1} & 0 & Y' \\  
0 & 0 & I_{n-t} & X' \\
  0 & 0 & 0 & I_t \\
\end{pmatrix}.$$ Now the normality follows from calculations
similar to \eqref{relations}.
\end{proof}
Using Mackey decomposition and the fact that the quotients
$${K_1}/(K_1\cap R(m))~{\rm and}~ {K_2}/(K_2\cap P(1,m+1))$$ are abelian, we
have
$${\res}_{K_1}\ind_{R(m)}^{P(1,m)}\id=\oplus_{\Lambda_1} \eta\
\text{and}\
{\res}_{K_2}\ind_{P(1,m+1)}^{R(m)}\id=\oplus_{\Lambda_2} \eta,$$
where $\Lambda_1$ and $\Lambda_2$ are characters on the quotients
${K_1}/(K_1\cap R(m))$ and ${K_2}/(K_2\cap P(1,m+1))$
respectively. The groups $P(1,m)$ and $R(m)$ act on $\Lambda_1$ and
$\Lambda_2$ respectively. We denote by $\Lambda_1'$ and $\Lambda_2'$
for a set of representatives for the action of $P(1,m)$ and $R(m)$
respectively. Now using Clifford theory, we obtain
\begin{equation}
\ind_{R(m)}^{P(1,m)}\id\simeq
\oplus_{\eta\in\Lambda_1'}\ind_{Z_{P(1,m)}(\eta)}^{P(1,m)}U_{\eta}
\end{equation}
and
\begin{equation}
\ind_{P(m+1)}^{R(m)}\id\simeq
\oplus_{\eta\in\Lambda_2'}\ind_{Z_{R(m)}(\eta)}^{R_{m}}U'_{\eta},
\end{equation}
where $U_{\eta}$ and $U'_{\eta}$ are some irreducible representations
of $Z_{P(1,m)}(\eta)$ and $Z_{R(m)}(\eta)$ respectively. The precise
description of $U_\eta$ is not used in any arguments. 

It is crucial to understand the images of the groups
  $Z_{P(1,m)}(\eta)$ and $Z_{R(m)}(\eta)$ in the quotient
  $K/K(1)$. This is achieved in Lemma \ref{lemma_bounds}, and we
begin with some preparations.  We first note that Iwahori
decomposition gives us
$$Z_{P(1,m)}(\eta)=Z_{P(1, m)\cap M_{1}}(\eta)K_1$$
and 
$$Z_{R(m)}(\eta)=Z_{R(m)\cap M_{1}}(\eta)K_2.$$
We have the following isomorphisms $$K_1/(K_1\cap R(m))\cong
\bar{\mathfrak{n}}_{1}^1$$ and $$K_2/(K_2\cap
P(1,m+1))\cong\bar{\mathfrak{n}}_{1}^2$$ respectively. The
$k_F$-dual of the space $\bar{\mathfrak{n}}_{1}^i$ is isomorphic to
$\bar{\mathfrak{n}}_{1}^i$ for $i\in\{1,2\}$, in a
$M_{1}(k_F)$-equivariant way. This is because the representation of
$M_{1}(k_F)$ on $\bar{\mathfrak{n}}_{1}^i$ is self-dual for
$i\in\{1,2\}$. Note that $P(1,m)\cap M_{1}=R(m)\cap M_{1}$. Observe
that the action of the groups $P(1,m)\cap M_{1}$ and $R(m)\cap
M_{1}$ on the characters in $\Lambda_1$ and $\Lambda_2$ factors
through the quotient map
\begin{equation}\label{reduction_levi}
\pi_1:K\cap M_{1}\rightarrow M_{1}(k_F).
\end{equation}
We
identify the group $M_{1}(k_F)$ with 
\begin{equation}\label{max_levi}
{\rm GL}(\bar{W}_r^+) \times G(\bar{V}^+_{r+1}\oplus \bar{V}^-_{r+1})
\end{equation} 
where $G(\bar{V}^+_{r+1}\oplus \bar{V}^-_{r+1})$ is the group of $k_F$-points
of the connected component of the isometry group of the pair
$(\bar{V}^+_{r+1}\oplus \bar{V}^-_{r+1}, \bar{q})$. The image of
$P(1,m)\cap M_{1}$ under the map \eqref{reduction_levi} is
contained in the group of the form
\begin{equation}\label{image_reduction}
Q\times G(\bar{V}^+_{r+1}\oplus \bar{V}^-_{r+1})
\end{equation}
 where $Q$ is the parabolic subgroup of
${\rm GL}(\bar{W}_r^+)$ fixing the flag $\bar{W}_1^+\subset
\cdots\subset \bar{W}_r^+$.

With the above observation, it is useful to  recall the
stabilisers in the case of general linear groups (see \cite[Lemma
3.8]{level_zero_gl_n_types}). Let $r>1$ be an integer and let
$I=(n_1,n_2,\cdots,n_r)$ be a partition of $n$. We denote by $P_I$,
the parabolic subgroup of upper block diagonal matrices of size
$n_i\times n_j$. The partition $(n_1, n_2, \cdots, n_{r-1})$ is
denoted by $J$.  Let $\mathcal{O}_{A}$ be an orbit for the action of
$P_{J}(k_F)\times\g{n_r}{k_F}$ on the set of matrices
$M_{(n-n_r)\times n_r}(k_F)$ given
by
$$(g_1,g_2)X=g_1Xg_2^{-1}~\forall~g_1\in
P_{J}(k_F),~g_2\in\g{n_r}{k_F},~X\in M_{(n-n_r)\times n_r}(k_F).$$ Let
$p_j$ be the composition of the quotient map
$P_J(k_F)\times {\rm GL}_{n_r}(k_F)\rightarrow M_I(k_F)$ and the
projection onto the $j^{th}$-factor of
$M_I(k_F)=\prod_{i=1}^{r}\g{n_i}{k_F}$ i.e.
$$p_j:P_J(k_F)\times {\rm GL}_{n_r}(k_F)\rightarrow {\rm GL}_{n_j}(k_F).$$
\begin{lemma}\label{lemma_level_4}
  Let $\mathcal{O}_A$ be an orbit consisting of non-zero matrices in
  $M_{(n-n_r)\times n_r}(k_F)$.  We can choose a representative $A$
  such that the $P_{J}(k_F)\times\g{n_r}{k_F}$-stabiliser
  $Z_{P_{J}(k_F)\times\g{n_r}{k_F}}(A)$ of $A$,
  satisfies one of the
  following conditions.
\begin{enumerate}
\item There exists a positive integer $j$ with $j\leq r$ such that the
  image of
$$p_j:Z_{P_{J}(k_F)\times\g{n_r}{k_F}}(A)\rightarrow \g{n_j}{k_F}$$
is contained in a proper parabolic subgroup of $\g{n_j}{k_F}$.
\item There exists a positive integer $i$ with $1\leq i\leq r-1$ such that
  $p_i(g)=p_r(g)$, for all $g$ in 
$$ Z_{P_{J}(k_F)\times\g{n_r}{k_F}}(A).$$
\end{enumerate}
\end{lemma} 
Now let us note a small observation which will be useful in the proof of 
Lemma \ref{lemma_bounds}.
\begin{lemma}\label{auto_unipotent}
Let $G$ be a split reductive group with an automorphism $\theta$. There
exists a parabolic subgroup of $G\times G$ with unipotent radical 
$U$ such that $\{(g, \theta(g))| g\in G\}$ has trivial intersection
with $U$.
\end{lemma}
\begin{proof}
Let $P$ be any proper parabolic subgroup of $G$ and $\bar{P}$
be any opposite parabolic subgroup of $P$. The unipotent radical 
of $P\times \bar{P}$ has trivial intersection with the diagonal subgroup
of $G\times G$. The group $\{(g, \theta(g))| g\in G\}$ is the image by
the automorphism $\id\times \theta$ of the diagonal subgroup of $G\times G$
and hence the lemma follows. 
\end{proof}
The following is the technical heart of this article. {\bf Here we use the
condition that $q_F>5$}. Let $\tilde{H}$ be the image of
$P(1,m)\cap M_{1}$ under the map $\pi_1$ in
\eqref{reduction_levi}. This is contained in the group $Q\times
G(\bar{V}^+_{r+1}\oplus \bar{V}^-_{r+1})$ as in
\eqref{image_reduction}. Hence the lemma is based on the $Q\times
G(\bar{V}^+_{r+1}\oplus \bar{V}^-_{r+1})$-stabilisers (which contain
$\tilde{H}$-stabilisers) of non-trivial elements in
$\bar{\mathfrak{n}}^1_1$ and $\bar{\mathfrak{n}}^2_1$. There are
several cases to consider primarily depending on the subspace
$\bar{W}_r^+$ of the flag $\bar{W}_1^+\subset \cdots\subset
\bar{W}_r^+$ being maximal or not. Let $\theta$ be the quotient map
$$\theta: Q\times G(\bar{V}_{r+1}^+\oplus 
\bar{V}_{r+1}^-)\rightarrow M(k_F).$$
\begin{lemma}\label{lemma_bounds}
  Let $u$ be any non-trivial element of $\bar{\mathfrak{n}}^1_{1}$ or
  $\bar{\mathfrak{n}}_{1}^2$ and $H$ be the image of
  $Z_{\tilde{H}}(u)$ under the map $\theta$. Let $\tau$ be a cuspidal
  representation of $M(k_F)$ and $\xi$ be an irreducible
  subrepresentation of $\res_{H}\tau$. There exists an irreducible
  representation $\tau'$ of $M(k_F)$ such that $\xi$ occurs in the
  restriction $\res_{H}\tau'$ and the inertial
    classes associated to the pairs $(P(1),\tau)$ and $(P(1),\tau')$
  are distinct.
\end{lemma}
\begin{proof}
  We will show that there exists a parabolic subgroup $S$ of $M(k_F)$
  such that $\Rad(S)\cap H$ is trivial. Using Lemma \ref{ref_disct} we
  get a non-cuspidal irreducible $M(k_F)$-representation $\tau'$ such
  that $\xi$ occurs in $\res_H\tau.$ The inertial classes
  associated to the pairs $(P_I(1), \tau)$ and $(P_I(1), \tau')$ are
  clearly distinct.
  
  We begin with the case where {\bf the space $W_{r}^+$ is a maximal
    isotropic subspace of $(W,q)$}. In this case, $P$ is contained in
  the maximal parabolic subgroup $P_1$ fixing the maximal isotropic
  subspace $W_r^+$ of $W$. Recall that the standard Levi factor of
  $P_1$ is denoted by $M_1$. The adjoint action of
  $M_1(k_F)\simeq {\rm GL}(\bar{W}_r^+)$ on $\bar{\mathfrak{n}}_{1}$,
  the Lie algebra of the unipotent radical of $\bar{P}_1(k_F)$, is
  the representation of ${\rm GL}(\bar{W}^+_r)$ on the space of
  $-\epsilon$ forms on $\bar{W}^+_{r}$. 

  Let $B$ be a $-\epsilon$ bilinear form on $\bar{W}_r^+$
  corresponding to $u$. In this case $\tilde{H}$ is contained in
  $Q$. Let $g=(g_{kl})$ and $B=(B_{k'l'})$ be the block matrix
  representation of the elements $g$ in $Q$ and the $-\epsilon$
  bilinear form $B$ on $\bar{W}_r^+$ with respect to the decomposition
  $\bar{V}_1^+\oplus\cdots\oplus \bar{V}_r^+$ of $\bar{W}_r^+$. Let
  $p$ be the largest positive integer such that the $B_{pq}$ is
  non-zero for some $1\leq q\leq r$. Let $q$ be the largest positive
  integer such that $B_{pq}\neq 0$. For any $g\in Z_Q(B)$ we have
$$g_{pp}B_{pq}g_{qq}^{T}=B_{pq}$$
where $B_{pq}$ is bilinear form on $\bar{V}^+_{p}\times \bar{V}^+_q$.
Without loss of generality assume that
$$\dim \bar{V}^+_{p}>\dim \bar{V}^+_{q}.$$ Let $S$ be the stabiliser of
the kernel of the map $\bar{V}^+_p\rightarrow (\bar{V}^+_q)^\vee$
induced by $B_{pq}$. Then $g_{pp}$ belongs to a proper parabolic
subgroup $\bar{S}$ of ${\rm GL}(\bar{V}^+_{p})$. Hence $H$ is
contained in a proper parabolic subgroup $\bar{S}$ of $M(k_F)$. The
required parabolic subgroup $S$ can be taken to be any opposite
parabolic subgroup of $\bar{S}$.

{\bf Consider the case where $\dim \bar{V}^+_{p}$ is equal to
  $\dim \bar{V}^+_{q}>1$}. If the map
$\bar{V}^+_{p}\rightarrow (\bar{V}^+_{q})^\vee$ induced by $B_{pq}$
has non-trivial kernel then $g_{pp}$ belongs to the proper parabolic
subgroup of ${\rm GL}(\bar{V}^+_{p})$ fixing this kernel. Hence
$H$ is contained in a proper parabolic subgroup $\bar{S}$ of
$M(k_F)$. Let $S$ be an opposite parabolic subgroup of $\bar{S}$. We
get that $\Rad(S)\cap H$ is a trivial group. We assume that the map
$\bar{V}^+_{p}\rightarrow (\bar{V}^+_{q})^\vee$, induced by $B_{pq}$,
is an isomorphism. Now using Lemma \ref{auto_unipotent}, we get a
proper parabolic subgroup $S$ of $M(k_F)$, with unipotent radical $U$,
such that $H\cap U$ is trivial.

We {\bf consider the case where $\dim \bar{V}^+_{p}$ is equal to
  $\dim \bar{V}^+_{q}=1$}. In this case, the group $H$ consists of
elements of the form
$$\diag(g_1, \cdots, g_p, \cdots , g_q, \cdots, g_r)$$
where $g_i\in {\rm GL}(\bar{V}_i^+)$ for $i\in\{p,q\}$ and
$g_pg_q=1$. We identify the representation $\tau$ with
$\tau_1\boxtimes\tau_2\boxtimes\cdots\boxtimes \tau_r$ where $\tau_i$
is a cuspidal representation of ${\rm GL}(\bar{V}^+_{i})$. Let $\eta$
be a non-trivial character of $k_F^\times$ and $\tau'$ be the
representation
$$\tau_1\boxtimes\cdots\boxtimes \tau_p\eta\boxtimes\cdots\boxtimes
\tau_q\eta^{-1}\boxtimes\cdots\boxtimes \tau_r.$$ Now the Bernstein
components associated to the pairs $(P_I(1), \tau)$ and
$(P_I(1), \tau')$ are the same if and only if the set
$\{\tau_p\eta, \tau_p^{-1}\eta^{-1}\}$ is either equals to
$\{\tau_p, \tau_p^{-1}\}$ or $\{\tau_q\eta^{-1},
\tau_p^{-1}\eta\}$. Hence, the character $\eta$ belongs to the set
$\{\tau_p^{-2}, \tau_p\tau_q, \tau_p\tau_q^{-1}\}$. Since $q_F>5$, we
can find a character $\eta$ such that $\eta$ does not belong to the
set $\{\tau_p^{-2}, \tau_p\tau_q, \tau_p\tau_q^{-1}\}$. For such a
choice of $\eta$ the Bernstein components associated to the pairs
$(P(1), \tau)$ and $(P(1), \tau')$ are distinct, and from construction
$\res_H\tau$ is equal to $\res_H\tau'$.
 
We come to the case when {\bf $\bar{W}_r^+$ is not a maximal isotropic
  subspace}. In this case, the space $\bar{V}_{r+1}$ is non-zero. The
standard Levi factor $M_{1}$ of $P_{1}$ is isomorphic to
$${\rm GL}(\bar{W}^+_r)\times G(\bar{V}_{r+1}).$$ 
Recall the notation $\bar{V}_{r+1}$ for the space
$(\bar{W}_r^+\oplus \bar{W}_r^-)^\perp$. The adjoint action of
$M_{1}$ on $\mathfrak{n}_{1}^2$ factors through the map
$${\rm GL}(\bar{W}_r)\times G(\bar{V}_{r+1})
\rightarrow {\rm GL}(\bar{W}_r).$$ In this case, the action of
${\rm GL}(\bar{W}_r)$ on $\mathfrak{n}_{1}^2$ is its representation
on the space of $-\epsilon$ forms. This case is similar to the case
where $\bar{W}^+_r$ is maximal and the proof of the lemma, in this case,
follows from the analysis in the previous case.

The action of $M_{1}(k_F)$ on $\mathfrak{n}_{1}^1\simeq {\rm
  Hom}(\bar{W}^+_r, \bar{V}_{r+1})$ is given by $$(g_1,
g_2)X=g_1Xg_2^{-1},~\forall~g_1\in {\rm GL}(\bar{W}^+_r),~g_2\in
G(\bar{V}_{r+1}).$$  We have to consider the stabilisers of $Q\times
G(\bar{V}_{r+1})$ on the space ${\rm
  Hom}(\bar{W}_r^+,\bar{V}_{r+1})$. Let $X$ be a non-zero element of
${\rm Hom}(\bar{W}_r^+, \bar{V}_{r+1})$. We have the decomposition 
$${\rm Hom}(\bar{W}_r^+,\bar{V}_{r+1})\simeq 
\bigoplus_{i=1}^r{\rm Hom}(\bar{V}_r^+,\bar{V}_{r+1}).$$ Now decompose
$X$ as the sum $\sum_{i=1}^rX_i$ such that $X_i$ belongs to ${\rm
  Hom}(\bar{V}_r^+,\bar{V}_{r+1})$. Let  $g=(g_{mn})$ be the block 
matrix form of any element in $Q$ with respect to the decomposition 
$$\bar{W}_r^+=\bar{V}_1^+\oplus\cdots\oplus \bar{V}_r^+.$$  
Let $t$ be the least positive integer such that $X_t$ is non-zero. We
then have $$g_{tt}X_{t}g^{-1}=X_t~\forall~g_{tt}\in{\rm
  GL}(\bar{V}_t^+),\tilde g\in G(\bar{V}_{r+1}).$$ Now let $R$ be the
group ${\rm GL}(\bar{V}_t^+)\times G(\bar{V}_{r+1})$.

{\bf Consider the case when
  $\dim(\bar{V}_t^+)>\dim(\bar{V}_{r+1})$}. In this case $Z_{R}(X_t)$ is
contained in a subgroup of the form $P\times G(\bar{V}_{r+1})$ where
$P$ is a proper parabolic subgroup of ${\rm GL}(\bar{V}_t^+)$ (see
Lemma \ref{lemma_level_4}).  Hence the unipotent radical of
$\bar{P}\times G(\bar{V}_{r+1})$, for any opposite parabolic subgroup
$\bar{P}$ of $P$, has trivial intersection with $Z_R(X_t)$. This shows that
there exists an unipotent radical of $M(k_F)$ which has trivial intersection
with $H$ and hence we get the lemma.

{\bf Now assume that $\dim(\bar{V}_t^+)$ is equal to
  $\dim(\bar{V}_{r+1})$}. In this case {\bf if the rank of $X_t$ is
  not equal to $\dim(\bar{V}_t^+)$} then $Z_R(X_t)$ is contained in
$P\times G(\bar{V}_{r+1}^+)$ where $P$ is a proper parabolic subgroup
of ${\rm GL}(\bar{V}_t^+)$ from similar arguments of the previous case
we prove the lemma. {\bf If the rank of $X_t$ is equal to
  $\dim(\bar{V}_t)$} then $Z_R(X_t)$ is contained in a group of the form
$$\{(X_tgX_t^{-1}, g); g\in  G(\bar{V}_{r+1}^+)\}.$$
Consider any Borel subgroup $B$ of ${\rm GL}({V}_{r+1}^+)$ such that
$B\cap G(\bar{V}_{r+1}^+)$ is the Borel subgroup of
$G(\bar{V}_{r+1}^+)$.  Let $\bar{B}$ be any opposite Borel
subgroup of $B$. The group $\bar{B}\times B$ can be identified with a
Borel subgroup of ${\rm GL}(\bar{V}_t^+)\times G(\bar{V}_{r+1})$. Now
the unipotent radical of the Borel subgroup
$X_t\bar{B}X_t^{-1}\times B$ has trivial intersection with
$Z_R(X_t)$, which proves the lemma in this case.

Let $(g_1, g_2)$ be an element of the group $Z_R(X_t)$ such that $g_1\in
{\rm GL}(\bar{V}_t^+)$ and $g_2\in G(\bar{V}_{r+1})$.  {\bf We are
  left with the case when
  $\dim(\bar{V}_t^+)<\dim(\bar{V}_{r+1})$}. Let $X_t\in {\rm
  Hom}_{k_F}(\bar{V}^+_t, \bar{V}_{r+1})$ be an operator such that
$\ker(X_t)$ is a non-zero subspace (since $X_t$ is non-zero operator,
$\ker(X_t)$ is not equal to $\bar{V}_r^+$). The group $Z_R(X_t)$ is
contained in a group of the form $P\times G(\bar{V}_{r+1})$ where $P$
is a parabolic subgroup of ${\rm GL}(\bar{V}_t^+)$ fixing
$\ker(X_t)$. This shows that $H$ is contained in a proper parabolic
subgroup of $M(k_F)$.  Now assume that $X_t$ is surjective. If
$\Rad(X_t\bar{V}^+_t)$ is a proper non-zero subspace of
$(X_t\bar{V}^+_t, \bar{q})$ then for any $(g_1, g_2)$ in
$Z_R(X_t)$ the element $g_2$ stabilises the space
$X_t\bar{V}_t^+$. This implies that $g_2$ stabilises the space
$\Rad(X_t\bar{V}_t^+)$. This shows that $g_2$ stabilises a proper
isotropic subspace and hence is contained in a proper parabolic
subgroup of $G(\bar{V}_{r+1})$.

Finally, consider the case where {\bf the space $X_t\bar{V}^+_t$ is
  either totally isotropic or non-singular}. If the space
$X_t\bar{V}_t^+$ is totally isotropic, then the element $g_2$ belongs
to a proper parabolic subspace of $G(\bar{V}_{r+1})$. If
$X_t\bar{V}_t^+$ is a non-singular space then the form $\bar{h}'$,
obtained by pulling $\bar{h}$ restricted to $X_t\bar{V}^+_t$ to
$\bar{V}^+_t$, is preserved by $g_1$. Hence $g_1$ belongs to
$G((\bar{V}_t^+, h'))$. In both the cases we can find a proper
parabolic subgroup $P$ of $ {\rm GL}_r(\bar{W}^+_r)\times
G(\bar{V}_{r+1})$ such that $Z_R(X_t)$ has trivial intersection with
$\Rad(P)$ and hence proving the lemma.
\end{proof}
\section{Classification of \texorpdfstring{$K$}{}-typical
  representations}
We need the following well known lemma (see \cite[Lemma
2.6]{level_zero_gl_n_types}). For the sake of next lemma consider any
parabolic subgroup $P$ of a reductive group $G$ with a Levi factor
$M$. Let $U$ be the unipotent radical of $P$. Let $\bar{U}$ be the
unipotent radical of the opposite parabolic subgroup of $P$ with
respect to $M$. Let $J_1$ and $J_2$ be two compact open subgroups of
$G$ such that $J_1$ contains $J_2$.  Suppose $J_1$ and $J_2$ both
satisfy Iwahori decomposition with respect to the pair $(P,
M)$. Assume
$$J_1\cap U=J_2\cap U~{\rm and}~ J_1\cap \bar{U}=J_2\cap \bar{U}.$$
Let $\lambda$ be an irreducible smooth representation of $J_2$ which
admits an Iwahori decomposition i.e. $J_2\cap U$ and $J_2\cap \bar{U}$
are contained in the kernel of $\lambda$.
\begin{lemma}\label{common_levi_induct}
  The representation $\ind_{J_2}^{J_1}(\lambda)$ is the extension of
  the representation $\ind_{J_2\cap M}^{J_1\cap M}(\lambda)$ such that
  $J_1\cap U$ and $J_1\cap \bar{U}$ are contained in the kernel of the
  extension.
\end{lemma}
Let us resume with the present case where $G$ is a split classical
group.  Let $\mathfrak{s}=[M, \sigma_M]_G$ be an
  inertial class such that $M\neq G$. Let $K_M$ be a
hyperspecial maximal compact subgroup of $M$. Let $\sigma_M$ be a
cuspidal representation of $M$ such that $\sigma_M$ contains a
depth-zero type of the form $(K_M, \tau_M)$. Let the hyperspecial
vertex in Bruhat--Tits building of $M$, corresponding to $K_M$, be
contained in the apartment corresponding to a maximal split torus $T$
(defined over $F$) of $M$. Such a torus $T$ is characterised by the
property that $K_M\cap T$ is the maximal compact subgroup of $T$ (see
\cite[2.6]{unrefined_minimal_types}).

Let $K$ be a hyperspecial maximal compact subgroup of $G$ such that
$K$ contains $K_M$. Let $T$ be a torus defined as in the above
paragraph. Now $K\cap T$ is the maximal compact subgroup of $T$. This
shows that $K$ is the parahoric subgroup of $G$ associated to a
hyperspecial vertex in the apartment corresponding to $T$. Let $B$ be
the standard basis of $W$ associated to $T$. There exists a self-dual
lattice chain $\Lambda$ such that $B$ is a splitting of $\Lambda$ and
$K=U_0(\Lambda)\cap G$.

Now the group $M$ is $K$-conjugate to a standard Levi subgroup defined
with respect to the basis $B$ and a flag $\mathcal{F}_I$ as defined in
\eqref{flags}, for some sequence of integers $I$ as defined in
\eqref{num_inv}. Hence, we may (and do) assume that $M$ is a standard
Levi subgroup corresponding to $\mathcal{F}_I$. Let $P$ be the
standard parabolic subgroup fixing the flag $\mathcal{F}_I$.  The
group $M$ is a Levi factor of $P$. Let $P(1)$ be the group
$K(1)(P\cap K)$. The representation $\tau_M$ extends as a
representation of $P(1)$ such that $P(1)\cap U$ and $P(1)\cap \bar{U}$
are contained in the kernel of this extension. With this we have the
following theorem:
\begin{theorem}\label{classification theorem}
  Let $\mathfrak{s}=[M, \sigma_M]_G$ be an inertial class such that
  $M\neq G$. Assume that $\sigma_M$ contains a depth-zero type of the
  form $(K_M, \tau_M)$, where $K_M$ is a hyperspecial maximal compact
  subgroup of $M$. Let $K$ be a hyperspecial maximal compact subgroup
  of $G$ containing $K_M$.  If $\tau$ is an $\mathfrak{s}$-typical
  representation of $K$, then $\tau$ is a subrepresentation of
  $\ind_{P(1)}^K\tau_M$.
\end{theorem}
\begin{proof}
  Let $P$ be the $G$ stabilizer of the flag 
$$\mathcal{F}_I=W_1^+\subset W_2^+\subset \cdots\subset W_r^+.$$
  Let $P_1$ be the
  $G$-stabiliser of the space $W_r^+$. Let $\mathcal{F}_J$ be the flag
  $$W_1^+\subset W_2^+\subset \cdots\subset W_{r-1}^+.$$ 
  Let $P_J$ be the parabolic subgroup of $G(W_{r}^+)$ fixing the
  flag $\mathcal{F}_J$. Let $M_J$ be the subgroup of ${\rm GL}(W_r^+)$
  fixing the decomposition
$$V_1^+\oplus V_2^+\oplus\cdots\oplus V_r^+.$$
The group $M_J$ is a Levi factor of the parabolic subgroup $P_J$.  We
recall that
$$M\simeq G_1\times G_2\times \cdots\times G_r\times G_{r+1},$$
where $G_i={\rm GL}(V^+_i)$, for $1\leq i\leq r$, and $G_{r+1}$ is the
$F$-points of the connected component of the isotropy subgroup of
$(V_{r+1}, q)$. 

We then identify $\sigma_M$ with
$\sigma_1\boxtimes\cdots\boxtimes\sigma_{r+1}$ where $\sigma_i$ is a
cuspidal representation of the group $G_i$, for all $1\leq i\leq r+1$.
Let $\tau_i$ be the unique $K\cap G_i$-typical representation occurring
in the cuspidal representation $\sigma_i$, for $1\leq i\leq r+1$. The
$K_M$ representation $\tau_M$ is isomorphic to the representation 
$$\tau_1\boxtimes\cdots\boxtimes\tau_r\boxtimes \tau_{r+1}.$$
From Lemma \ref{conley_lemma} we know that any
irreducible $K$-subrepresentation of
$$i_{P}^{G}\sigma_M/\ind_{P\cap K}^K\tau_M$$ is atypical.
Now the representation $\ind_{P\cap K}^K\tau_M$ is the union of the
representations $\ind_{P(m)}^K\tau_M$ for $m\geq 1$.

Let $K'$ be the compact open subgroup ${\rm GL}(W_r^+) \cap K$ of
${\rm GL}(W_r^+)$. Let $K'(m)$ be the principal congruence subgroup of 
level $m$ contained in $K$. The compact group $K'(m)\cap (P_J\cap K')$
is denoted by $P_J(m)$. Let $\tau_J$ be the $K'\cap M_J$-representation 
$$\tau_1\boxtimes\tau_2\boxtimes\cdots\boxtimes\tau_r.$$
The representation $\tau_J$ extends as a representation of $P_J(m)$ via
inflation from the map 
$$P_J(m)\rightarrow P_J(k_F)\rightarrow M_J(k_F).$$

From transitivity of induction and using Lemma
\ref{common_levi_induct}, we see that
$$\ind_{P(m)}^K\tau_M\simeq
\ind_{P_{1}(m)}^{K}\{(\ind_{P_{J}(m)}^{K'}\tau_J)\boxtimes\tau_{r+1}\}.$$
The irreducible $K'$-subrepresentations of
$\ind_{P_{J}(m)}^{K'}\tau_J/\ind_{P_{J}(1)}^{K'}\tau_J$ are atypical
from the result \cite[Theorem 1.1]{level_zero_gl_n_types}. Hence
$\mathfrak{s}$-typical representations of $K$ can only occur as
subrepresentations of
$$\ind_{P_{1}(m)}^{K}\{(\ind_{P_{J}(1)}^{K'}\tau_J)\boxtimes\tau'\}\simeq
\ind_{P(1,m)}^{K}\tau_M.$$ 

Now from Lemmas 3.2 and 2.5 we get that
$$\ind^{P(1,m)}_{P(1,m+1)}\id=\id\oplus \bigoplus_{i=1}^{k}\ind_{H_i}^{P(1,m)}U_i$$
such that any irreducible subrepresentation $\chi$ of
$\res_{H_i}\tau_I$ occurs in $\res_{H_i}\tau_I'$. Moreover, the
Bernstein components associated to the pairs $(P_I(1), \tau_I)$ and
$(P_I(1), \tau_I')$ are distinct. 
Note that 
\begin{align*}
\ind_{P(1,m+1)}^K\tau_{M}\simeq&
\ind_{P(1,m)}^K\{\ind_{P(1,m+1)}^{P(1,m)}\id\}\otimes \tau_M\\
\simeq &\ind_{P(1,m)}^K\tau_{M}\oplus
         \ind_{H_i}^{P(1,m)}(U_i\times \res_{H_i}\tau_M).
\end{align*}
Using induction on $m$, any $\mathfrak{s}$-typical representation
occurs as a subrepresentation of
$\ind_{P(1)}^{K}\tau_M$. Recall that the subgroup
  $P(1,1)$ is equal to $P(1)$. Since $(P(1), \tau_M)$ is a
Bushnell--Kutzko's type for $[M, \sigma_M]$, we complete the proof of
the theorem.
\end{proof}
\section{Principal series components}\label{princ}
Let ${\bf G}$ be the split classical group defined as the connected
component of the isometry group of $(W,q)$, as in Section
\ref{preli}. Let $K$ be a hyperspecial maximal compact subgroup of
$G$. Let ${\bf T}$ be a maximal split torus of ${\bf G}$ defined over $F$
such that $K\cap T$ is the maximal compact subgroup of $T$. Let
\begin{equation}\label{princ_basis_std}
  (w_i :-n\leq i\leq n)
\end{equation}
be a standard basis associated to $T$. Now there exists a self-dual
lattice chain $\Lambda$ such that the basis \eqref{princ_basis_std} is
a splitting of $\Lambda$ and $K=U_0(\Lambda)\cap G$. Let
$$\Lambda(0)=\mathfrak{p}_F^{a_n}w_n\oplus 
\mathfrak{p}_F^{a_{n-1}}w_{n-1}\oplus\cdots\oplus
\mathfrak{p}_F^{a_{-n+1}}w_{-n+1}\oplus\mathfrak{p}_F^{a_{-n}}w_{-n}.$$
We fix a basis
$$\{\varpi_F^{a_n}w_n, \varpi_F^{a_{n-1}}w_{n-1},\dots,
\varpi_F^{a_{-n+1}}w_{-n+1}, \varpi_F^{a_{-n}}w_{-n} \}$$ of $W$. Now, using
this basis, we get an embedding
\begin{equation}\label{iota_embedding}
\iota :G\rightarrow \g{N}{F}.
\end{equation}
of $G$ in ${\rm GL}_N(F)$.
The image of the maximal compact subgroup $K$ can be identified with
${\rm GL}_N(\mathfrak{o}_F)\cap \iota(G)$. The torus $T$ is the group of
diagonal matrices of $\iota(G)$. Let ${\bf B}$ be the Borel
subgroup of ${\bf G}$ such that $B$ is a subgroup of upper triangular
matrices in ${\rm GL}_N(F)$. We denote by $\bar{{\bf B}}$, the
opposite Borel subgroup of ${\bf B}$ with respect to ${\bf T}$. Let
${\bf U}$ and $\bar{{\bf U}}$ be the unipotent radicals of ${\bf B}$ and
$\bar{{\bf B}}$ respectively.  

We identify the torus $T$ with
$(F^{\times})^n$ by the map
$$\diag(t_1, t_2, \dots t_n, t_n^{-1}, \dots,
t_2^{-1}, t_1^{-1})\mapsto (t_1, \cdots, t_n),~ t_i\in F^{\times}.$$
We also identify a character $\chi$ of $T$ with 
$$\chi=\chi_1\boxtimes\dots\boxtimes\chi_n,$$ where $\chi_i$ is a character
of $F^{\times}$. The conductor of $\chi_i$, denoted
by $l(\chi_i)$, is the least positive integer $n$
such that $1+\mathfrak{p}_F^n$ is contained in the kernel of
$\chi$. {\bf In this section, we assume that
  $$l(\chi_i)\neq l(\chi_j)~ {\rm for~all}~ i\neq j.$$}

Let $\mathfrak{s}$ be the inertial class $[T, \chi]$.  Let $\tau$ be
an $\mathfrak{s}$-typical representation of $K$. The representation
$\tau$ occurs as a subrepresentation of an irreducible smooth
representation $\pi$ of $G$. By definition, the inertial support of
the representation $\pi$ is equal to $\mathfrak{s}$. Hence, $\tau$ is
an irreducible subrepresentation $\res_Ki_B^G\chi$. The
$G$-representations $i_B^G\chi$ and $i_B^G\chi^w$ have the same
Jordan--Holder factors, for all $w\in N_G(T)$.  This shows that, for
the purpose of understanding $\mathfrak{s}$-typical representations of
$K$, we may (and do) arrange the characters
$\chi_1, \chi_2, \dots, \chi_n$ (conjugating by an element in the Weyl
group if necessary) such that
 \begin{equation}\label{conductor_condition}
l(\chi_i)>l(\chi_j)~{\rm for}~i<j.
\end{equation}

Types for any Bernstein component $[T, \chi]$ of a split reductive
group ${\bf G}$ are constructed by Roche in \cite{AllanRoche}.  We
recall his constructions from \cite[Section 2,3]{AllanRoche}. Let
${\bf B}$ be any Borel subgroup of ${\bf G}$ containing a maximal
split torus ${\bf T}$.  Let ${\bf U}$ be the unipotent radical of
${\bf B}$ and $\bar{{\bf U}}$ be the unipotent radical of the opposite
Borel subgroup $\bar{{\bf B}}$ of ${\bf B}$ with respect to ${\bf T}$.
Let $\Phi$ be the set of roots of ${\bf G}$ with respect to ${\bf T}$.
Let $\Phi^+$ and $\Phi^-$ be the set of positive and negative roots
with respect to the choice of the Borel subgroup ${\bf B}$
respectively. Let $f_{\chi}$ be the function on $\Phi$ defined by

\begin{eqnarray}
  f_{\chi}(\alpha)=\begin{cases}
    [l(\chi\alpha^\vee)]/2 &\text{if $\alpha\in \Phi^+$}\\
    [(l(\chi\alpha^\vee)+1)/2] &\text{if $\alpha\in \Phi^-$.}
      \end{cases}
 \end{eqnarray}
 Let $x_\alpha:\mathbb{G}_a\rightarrow {\rm U}_\alpha$ be the root
 group isomorphism, and let $U_{\alpha,t}$ be the group
 $x_\alpha(\mathfrak{p}_F^t)$.  Let $T_0$ be the maximal compact
 subgroup of $T$. Let $U_{\chi}^\pm$ be the group generated by
 $U_{\alpha, f_\chi(\alpha)}$, for all $\alpha\in \Phi^\pm$. Let
 $J_\chi$ be the group generated by $U_\chi^+$, $T_0$, and
 $U_\chi^-$. The group $J_\chi$ has Iwahori decomposition with respect
 to the pair $(B, T)$ such that
$$J_\chi\cap U=U_\chi^+,~J_\chi\cap \bar{U}=U^-_\chi,~{\rm
  and}~J_\chi\cap T=T_0.$$ The representation $\chi$ of $T_0$ extends
to a representation of $J_\chi$ such that $U_\chi^+$ and $U_\chi^-$
are both contained in the kernel of this extension. We use the same
notation $\chi$ for this extension.  The pair $(J_\chi, \chi)$ is
  a type for the Bernstein component $[T, \chi]$.  We apply these
results to a split classical group ${\bf G}$ with the diagonal torus
$T$ and the Borel subgroup ${\bf B}$ of ${\bf G}$ whose $F$-points are
upper triangular matrices, to get a type $(J_\chi, \chi)$ for $s$.
Let $\mathcal{I}$ be the group $K(1)(B\cap K)$. The group
$\mathcal{I}$ is an Iwahori subgroup of $G$, contained in $K$. We
  may (and do) choose the set of root group isomorphisms
  $\{x_\alpha:\mathbb{G}_a\rightarrow {\rm U}_\alpha|\ \alpha\in
  \Phi\}$ such that $J_{\rm \id}$ is equal to $\mathcal{I}$.  Moreover,
for such a choice, we get that $J_\chi$ is a subgroup of
$\mathcal{I}$.

Before going any further, we need some notation. Consider the isotropic
space $W_1^+$ spanned by $w_1$, and $W_1^-$ the space spanned by
$w_{-1}$. Let $P_1$ be a parabolic subgroup of $G$ fixing the space
$W_1^+$. Let $M_1$ be the standard Levi factor of $P_1$, i.e, the
$G$-stabiliser of the decomposition
$$W_1^+\oplus(W_1^+\oplus W_1^-)^\perp \oplus W_{1}^-.$$
The group $M_{1}$ isomorphic to $F^\times \times G(W')$, where $W'$ is
equal to $(W_1^+\oplus W_1^-)^\perp$. Let $\bar{U}_1$ be the unipotent
radical of the opposite parabolic subgroup $\bar{P}_1$ of $P_1$ with
respect to $M_1$. Let $m$ be any positive integer such that $m\geq
l(\chi_1)$.  Define the compact open subgroups $P_{1}^0(m)$ and
$R^0(m)$ by
$$P_{1}^0(m) =(U_{1}\cap P_{1}(m))(M_{1}\cap J_\chi)
 (\bar{U}_{1}\cap P_{1}(m))$$
and 
$$R^0(m) =(U_{1}\cap R(m))(M_{1}\cap J_\chi)(\bar{U}_{1}\cap R(m))$$
respectively. Here $R(m)$ is the group as defined in Section
\ref{sec_auxiliary}.

For inductive arguments we will use the decomposition of the following
representations 
$$\ind_{R^0(m)}^{P^0_{1}(m)}\id~{\rm and}~\ind^{R^0(m)}_{P^0_{1}(m+1)}\id.$$
 Let
$K_1$ and $K_2$ be the kernels of the maps
$$P_{1}^0(m)\xrightarrow{\pi_1} P_{1}(k_F)\rightarrow 
M_{1}(k_F)~{\rm and}~R^0(m)\xrightarrow{\pi_1} P_{1}(k_F)\rightarrow
M_{1}(k_F)$$ respectively.  Recall that the map $\pi_1$ is reduction
mod $\mathfrak{p}_F$ map. Using the arguments similar to Lemma
\ref{normality_lemma} we get that
$$K_1\cap R^0(m)\trianglelefteq K_1~{\rm and}~K_2\cap
P_{1}^0(m+1)\trianglelefteq K_2.$$ Now let $\Lambda_1$ and
$\Lambda_2$ be the set of representatives for the orbits of the action
of the groups $P^0_{1}(m)$ and $R^0(m)$ on the set of characters of
the groups $K_1/(K_1\cap R^0(m))$ and $K_2/(K_2\cap
P^0_{1}(m+1))$. We then have
$$\ind^{P^0_{1}(m)}_{R^0(m)}\id
\simeq \oplus_{\eta\in \Lambda_1}
\ind_{Z_{P^0_{1}(m)}(\eta)}^{P^0_{1}(m)}U_\eta$$
and 
$$\ind^{R^0(m)}_{P_{1}(m+1)}\id
\simeq \oplus_{\eta\in \Lambda_2}
\ind_{Z_{R^0(m)}(\eta)}^{R^0(m)}U_\eta.$$
We note that $$Z_{P^0_{1}(m)}(\eta)=Z_{P^0_{1}(m)\cap M_{1}}(\eta)
K_1~{\rm and}~Z_{R^0(m)}(\eta)=Z_{R^0(m)\cap M_{1}}(\eta) K_2.$$ 

The group of characters of $K_1/(K_1\cap R^0(m))$ and $K_2/(K_2\cap
P^0_{1}(m+1))$ are isomorphic to the groups
$\bar{\mathfrak{n}}_{1}^1$ and $\bar{\mathfrak{n}}_{1}^2$
respectively.  The action of the group $P_{1}^0(m)\cap
M_{1}=R^0(m)\cap M_{1}$ factors through the quotient map
\begin{align*}
P_{1}^0(m)\cap M_{1}\rightarrow  M_{1}(k_F).
\end{align*}
The image 
of this quotient map is contained in $B(k_F)\cap M_{1}(k_F)$.
\begin{lemma}\label{lemma_princ}
  Let $u$ be any non-trivial element of $\bar{\mathfrak{n}}_1^i$ for
  $i\in\{1,2\}$. Let $H$ be the group $Z_{M_{1}(k_F)\cap
    B(k_F)}(u)$. There exists a character $\chi'$ of $T$ such that
  $$\res_H\chi=\res_H\chi'$$ and the inertial classes $[T, \chi]$
  and $[T, \chi']$ are distinct.
\end{lemma}
\begin{proof}
  The group $M_{1}(k_F)\cap B(k_F)$ is isomorphic to $k_F^\times
  \times B'$, where $B'$ is a Borel subgroup of $G(\bar{W}',
  \bar{q})$. The action of the group $k_F^\times\times B'$ on
  $\bar{\mathfrak{n}}_1^2$ factors through the projection
	$$k_F^\times \times B'\rightarrow k_F^\times.$$
	The action is given by the character $x\mapsto x^2$. Hence if
        $(x, b)$ belongs to $Z_{k_F^\times \times B'}(u)$ where $u\in
        \bar{\mathfrak{n}}_{1}^1\backslash \{0\}$ then $x^2=1$. In
        this case, consider a non-trivial character $\eta$ of
        $k_F^\times$ which is trivial on the group $\{\pm 1\}$. We
        consider the character $\eta$ as a character of
        $\mathfrak{o}_F^\times$ via inflation.  Set $\chi'$ to be the
        character
        $\chi_1\eta\boxtimes\chi_2\boxtimes\cdots\boxtimes\chi_n$. From
        the above definition we get $$\res_H\chi=\res_H\chi'.$$ If the
        Bernstein component $[T, \chi_1]$ is equivalent to $[T,
        \chi_2]$ then $\eta^{-1}=\chi_1^2$.  This is not possible as
        $l(\chi_1)\neq 1$. Hence the character $\chi'$ is the
        character satisfying the lemma.
	
	Now consider the case when $u$ belongs to
        $\bar{\mathfrak{n}}_{1}^1$.  The unipotent radical $U$ of
        $k_F^\times \times B'$ is a $p$-group. Hence there exists a
        flag $\{V_i; V_i\subset V_{i+1}\}$ of
        $\bar{\mathfrak{n}}_{1}^1$ stabilised by $k_F^\times \times
        B'$ such that $U$ acts trivially on $V_i/V_{i+1}$. Let $i$ be
        the least positive integer such that $u\in V_i$. The group $H$
        is contained in the $k_F^\times \times B'$-stabiliser of
        $\bar{u}$ in $V_i/V_{i-1}$. The group $U$ acts trivially on
        $V_i/V_{i-1}$. Hence the image of $H$ under the natural map
        $k_F^\times \times B' \rightarrow T(k_F)$ is contained in a
        group of the form
	$$\{\diag(t_1,t_2,\cdots, t_n, 1, t_{-n}, \cdots, t_1)|\  
	t_1t_j^{-1}=1\}.$$ Without loss of generality, assume that
        $j>0$. Consider the character $\chi'$ given by
	$$\chi'=\chi_1\eta
	\boxtimes\cdots\boxtimes\chi_j\eta^{-1}\boxtimes\cdots\boxtimes
        \chi_n.$$ If $(T, \chi)$ and $(T, \chi')$ are inertially
        equivalent, then the multiplicity of $\{\chi_1, \chi_1^{-1}\}$
        in the following multi-sets
	$$\{\{\chi_1, \chi_1^{-1}\}, \cdots, \{\chi_n, \chi_n^{-1}\}\}$$
	and  
	$$\{\{\chi_1\eta, \chi_1^{-1}\eta^{-1}\}, \cdots, 
	\{\chi_j\eta^{-1}, \chi_j^{-1}\eta\}, \cdots, 
	\{\chi_n, \chi_n^{-1}\}\}$$ must be the same.
	This implies that $\eta$ belongs to 
	$\{\chi_1^{-2},\chi_1\chi_j, \chi_1\chi_j^{-1}\}$. 
	Since $k_F^\times$ has cardinality bigger than $5$, there exists a 
	character $\eta$ such that $[T, \chi]$ and $[T, \chi']$ are not 
	inertially equivalent. This completes the proof of the lemma. 
      \end{proof}
      We need the following technical observation. Let $\chi$ and
      $\eta$ be two characters of $T$. Recall that $T$ is identified
      with $(F^\times)^n$ using the diagonal embedding using $\iota$ in
      \eqref{iota_embedding}. We identify $\chi$ with
      $\boxtimes_{i=1}^n\chi_i$ and $\eta$ with
      $\boxtimes_{i=1}^n\eta_i$.
      \begin{lemma}\label{princ_multip}
        Let $n>1$, and let $[T, \chi]_{M_1}$ and $[T, \eta]_{M_1}$ be two inertial
        classes such that $\res_{\mathfrak{o}_F^\times}\chi_1=\res_{\mathfrak{o}_F^\times}\eta_1$. If
        $[T, \chi]_{M_1}\neq [T, \eta]_{M_1}$, then
        $[T, \chi]_G\neq [T, \eta]_G$
      \end{lemma}
      \begin{proof} 
       Since $[T, \chi]_{M_1}\neq [T,
        \eta]_{M_1}$, there exists an integer $i$ with $2\leq i\leq n$
        such that the multiplicity of the multiset
        $\{\res_{\mathfrak{o}_F^\times}\chi_i,
        \res_{\mathfrak{o}_F^\times}\chi_i^{-1}\}$ has different multiplicities in the multisets
        $$\{\{\res_{\mathfrak{o}_F^\times}\chi_2,
        \res_{\mathfrak{o}_F^\times}\chi_2^{-1}\},\dots,
        \{\res_{\mathfrak{o}_F^\times}\chi_n,
        \res_{\mathfrak{o}_F^\times}\chi_n^{-1}\}\}$$
        and
        $$\{\{\res_{\mathfrak{o}_F^\times}\eta_2,
        \res_{\mathfrak{o}_F^\times}\eta_2^{-1}\},\dots,
        \{\res_{\mathfrak{o}_F^\times}\eta_n,
        \res_{\mathfrak{o}_F^\times}\eta_n^{-1}\}\}.$$ Hence, the
        multiset
        $\{\res_{\mathfrak{o}_F^\times}\chi_i,
        \res_{\mathfrak{o}_F^\times}\chi_i^{-1}\}$ will have different
        multiplicities in
        $$\{\{\res_{\mathfrak{o}_F^\times}\chi_1,
        \res_{\mathfrak{o}_F^\times}\chi_1^{-1}\},
        \{\res_{\mathfrak{o}_F^\times}\chi_2,
        \res_{\mathfrak{o}_F^\times}\chi_2^{-1}\},\dots,
        \{\res_{\mathfrak{o}_F^\times}\chi_n, \res_{\mathfrak{o}_F^\times}\chi_n^{-1}\}\}$$
        and
        $$\{\{\res_{\mathfrak{o}_F^\times}\eta_1,
        \res_{\mathfrak{o}_F^\times}\eta_1^{-1}\},
        \{\res_{\mathfrak{o}_F^\times}\eta_2,
        \res_{\mathfrak{o}_F^\times}\eta_2^{-1}\},\dots,
        \{\res_{\mathfrak{o}_F^\times}\eta_n,
        \res_{\mathfrak{o}_F^\times}\eta_n^{-1}\}\}.$$ This shows the lemma.
        \end{proof}
        We are now ready to classify $\mathfrak{s}=[T, \chi]$-typical
        representations of $K$.
\begin{theorem}\label{main prince}
  Let $K$ be the fixed hyperspecial maximal compact subgroup $G$. Let
  $\mathfrak{s}=[T, \boxtimes_{i=1}^n\chi_i]_G$ be a
  toral inertial class such that
  $l(\chi_i)>l(\chi_i)$, for all $i<j$. If $\tau$ is an
  $\mathfrak{s}$-typical representation of $K$, then $\tau$ is a
  subrepresentation of $\ind_{J_\chi}^K\chi$.
\end{theorem}
\begin{proof}
  Using induction on $n$ we show that the representation
  $\ind_{J_\chi}^K\chi$ is a subrepresentation of $\res_Ki_B^G\chi$,
  and any irreducible subrepresentation
  of $$(\res_Ki_B^G\chi)/\ind_{J_\chi}^K\chi$$ is atypical.

  Assume
  this hypothesis to be true for all $n'<n$. From induction
  hypothesis, we get that 
  $$\res_Ki_{B\cap M_1}^{M_1}\chi=\ind_{J_\chi\cap M_1}^{K\cap M_1}\chi\oplus \tau'$$
  such that any irreducible $(K\cap M_1)$-subrepresentation of $\tau'$
  is atypical. Let $\xi$ be a $(K\cap M_1)$-irreducible
  subrepresentation of $\tau'$. Since the $(K\cap M_1)$-representation
  $\xi$ is atypical, it occurs as a subrepresentation of
  $\res_{K\cap M_1}i_S^{M_1}\kappa$, where $S$ is a standard parabolic
  subgroup of $M_1$ with Levi factor $L$ and $\kappa$ is cuspidal
  representation of $L$ such that
  $[L, \kappa]_{M_1}\neq [T, \chi]_{M_1}$.  Any irreducible
  $K$-subrepresentation of $\ind_{K\cap P_1}^{K_1}\xi$ occurs as a
  $K$-subrepresentation of
\begin{equation}\label{last_conley_lemma_rep}
  i_{P_1}^G(i_{S}^{M_1}\kappa).
  \end{equation}
  If $L\neq T$, then the cuspidal support of the representation
  \eqref{last_conley_lemma_rep} is not equal to $[T, \chi]_G$. Assume
  that $L=T$. Since we have $[T, \kappa]_{M_1}\neq [T, \chi]_{M_1}$,
  using Lemma \ref{princ_multip}, we get that
  $[T, \kappa]_G\neq [T, \chi]_G$. Hence, the irreducible
  subrepresentations of $\ind_{K\cap P_1}^{K_1}\xi$ are atypical.

  Let $\tau$ be any $\mathfrak{s}$-typical representation of $K$. From
  the above discussion, we get that $\tau$ is a subrepresentation of
	\begin{equation}\label{candidate}
          \ind_{K\cap P_{1}}^K\gamma~{\rm with}~\gamma=\ind_{J_\chi\cap
          M_{1}}^{K\cap M_{1}}\chi.
	\end{equation}
        Now let $N$ be the integer $l(\chi_1)$, the largest among the
        set of integers $\{l(\chi_i): 1\leq i\leq n\}$. Now the
        representation \eqref{candidate} is the union of the
        representations $\ind_{P_{1}(m)}^K\gamma$ for $m\geq N$. Hence
        any $\mathfrak{s}$-typical representation of $K$ occurs as a
        subrepresentation of $\ind_{P_{1}(m)}^K\gamma$, for some
        $m\geq N$. Note that the representation
        $\ind_{P_{1}(m)}^K\gamma$ is isomorphic to the representation
        $\ind_{P_{1}^0(m)}^K\chi$ (see Lemma
        \ref{common_levi_induct}).
	
	We use induction on $m\geq N$ to show that irreducible
        subrepresentations of
	$$\ind_{P_{1}^0(m+1)}^K\chi/\ind_{P_{1}^0(m)}^K\chi$$
	are atypical for all $m\geq N$.
	Now we have the isomorphism 
	\begin{align*}
          \ind_{P_{1}^0(m+1)}^K\chi&\simeq
          \ind_{P_{1}^0(m)}^K\{\chi\otimes(\ind_{P_{1}^0(m+1)}^{P_{1}^0(m)}\id)\}\\
          &\simeq \ind_{P_{1}^0(m)}^K\chi\oplus_{\eta\in \Lambda_1}
          \ind_{Z_{P_{1}^0(m)}(\eta)}^K(\chi\otimes U_\eta)\\
          &\oplus_{\eta\in \Lambda_2}
          \ind_{Z_{R^0(m)}(\eta)}^K(\chi\otimes U_\eta).
\end{align*}
Using Lemma \ref{lemma_princ}, we obtain a character $\chi'$ such that
$\res_{H}\chi'$ is equal to $\res_H\chi$ where $H$ is either
$Z_{P_{1}^0(m)}(\eta)$ or $Z_{R^0(m)}(\eta)$. Moreover, $[T, \chi]$
and $[T, \chi']$ are distinct inertial classes. Hence, $\tau$ is
contained in the representation $\ind_{P_{1}^0(N)}^K\chi$.

Let $\mathcal{I}$ be the Iwahori subgroup $K(1)(B\cap K)$, we have
$J_\chi\subseteq \mathcal{I}$.  Using support of the $G$-intertwining of
the pair $(J_\chi, \chi)$ in \cite[Theorem 4.15]{AllanRoche}, we note
that the representation $\ind_{J_\chi}^{\mathcal{I}}\chi$ is
irreducible. Moreover, we have
$${\rm Hom}_{\mathcal{I}}(\ind_{J_\chi}^{\mathcal{I}}\chi, 
\ind_{P^0_{1}(N)}^{\mathcal{I}}\chi)\neq 0.$$ From the definition of
$J_\chi$, we note that the dimensions of the representations
$\ind_{J_\chi}^{\mathcal{I}}\chi$ and
$\ind_{P_1^0(N)}^{\mathcal{I}}\chi$ are the same. This shows that
these representations are isomorphic. We conclude that, for any
$\mathfrak{s}$-typical representation $\tau$ of $K$, we get that
$\tau$ is a subrepresentation of $\ind_{J_\chi}^K\chi$. Moreover, the
representation $\ind_{J_\chi}^K\chi$ is a subrepresentation of
$\res_{K}i_B^G\chi$.
\end{proof} {\bf Acknowledgements:} The present work was carried out
at TIFR, Mumbai when the authors were postdoc fellows.  We thank the
School of Mathematics for their hospitality. The authors express their
gratitude to Radhika Ganapathy for her help with Bruhat--Tits
theory. We also thank Anne-Marie Aubert for patiently listening to our
ideas. The first named author is funded by ISRAEL SCIENCE FOUNDATION
grant No. 421/17 (Mondal). The second named author is supported by the
NWO Vidi grant ``A Hecke algebra approach to the local Langlands
correspondence" (nr. 639.032.528).  We would like to thank the
anonymous referee for helpful suggestions.

\bibliographystyle{amsalpha}
\bibliography{./biblio}
\noindent
Santosh Nadimpalli, 
IMAPP, Radboud
Universiteit Nijmegen,
 Heyendaalseweg 135, 6525AJ Nijmegen, 
The
Netherlands.
\texttt{nvrnsantosh@gmail.com}, \texttt{Santosh.Nadimpalli@.ru.nl}.
\\
\\
Amiya Kumar Mondal,
Department of Mathematics, Bar-Ilan University,
Ramat Gan 529002, Israel.
\texttt{E-mail:amiya96@gmail.com}
\end{document}